\documentclass[english]{article}
\usepackage[T1]{fontenc}
\usepackage[utf8x]{inputenc}
\usepackage{graphicx}
\usepackage{amssymb}
\usepackage{amsmath}
\usepackage{amsthm}
\usepackage{color}
\usepackage{relsize}
\usepackage[all]{xy}
\usepackage{amsfonts}
\usepackage{mathrsfs}
\usepackage{latexsym}
\usepackage{geometry} 
\usepackage{textcomp}
\usepackage{todonotes}
  \usepackage[english]{babel}
\usepackage{tikz}
\usepackage{tikz-cd}
\usetikzlibrary{matrix,arrows}
\usepackage{cleveref}

\newcommand\ZZ{\mathbb{Z}} 
\newcommand\QQ{\mathbb{Q}} 

\newcommand\fleche{\longrightarrow}

\DeclareMathOperator{\Isom}{Isom}

\DeclareMathOperator{\Spec}{Spec}

\DeclareMathOperator{\Vect}{Vect}
\DeclareMathOperator{\id}{id}

\DeclareMathOperator{\Gr}{Gr}

\DeclareMathOperator{\GL}{GL}
\DeclareMathOperator{\Char}{Char}

\DeclareMathOperator{\hasse}{hasse}

\DeclareMathOperator{\rk}{rk}

\theoremstyle{definition} 
\newtheorem{defin}{Definition}[section]
\newtheorem{hypothese}[defin]{Hypothesis}

\theoremstyle{plain} 
\newtheorem{theor}[defin]{Theorem}
\newtheorem{lemm}[defin]{Lemma}  
\newtheorem{prop}[defin]{Proposition}

\newtheorem{cor}[defin]{Corollary}

\theoremstyle{remark} 
\newtheorem{rema}[defin]{Remark}
\newtheorem{example}[defin]{Example}

\newtheorem{claim}[defin]{Claim}

\definecolor{vert}{rgb}{0.09,0.62,0.40}

\begin{document}

\title{On the geometry of the Pappas-Rapoport models in the (AR) case}
\author{Stéphane Bijakowski and Valentin Hernandez}
\date{}
\maketitle

\section{Introduction}

In the last 50 years at least, Shimura varieties have played a central role in the Langlands program. Most of the time, these varieties can be thought of as moduli spaces of abelian varieties over $\Spec(\QQ)$. It turns out that for arithmetic applications it is sometimes desirable to have an integral structure on these spaces to be able to use their reduction modulo primes $p$. Such integral structures have been studied first by Deligne-Rapoport (\cite{DR}) and Katz-Mazur (\cite{KM}) for the modular curve and have been largely studied since then. 

Obviously, there are a priori a lot of possible choices for these integral models, but there is a natural way to choose 
one by extending the moduli problem over $\Spec(\ZZ)$ (or some localisation of it). This \textit{natural} strategy has 
been extensively studied and most of the results in the P.E.L. case can be found in works of Lan (\cite{Lan}) as long as the moduli problem is \textit{unramified}. An extra difficulty appears when we allow 
ramification in the moduli problem. In this case it has been realized a long time ago in works of Pappas and Rapoport 
(see e.g. \cite{PRLMI}) that the natural moduli problem has bad geometric properties. In their work, Pappas and 
Rapoport suggested to study a slightly different integral model than the \textit{natural} one, by adding to the moduli 
problem (parametrizing abelian schemes) an extra linear data of a flag of the Hodge filtration (with some restricting 
properties). This model is refered to as the \textit{splitting model}, or (as we call it) Pappas-Rapoport model. In some 
sense, this model should be thought of as a blow-up of the natural model along some of its singularities. The goal of this article is to study the geometry of this model, and in particular of its special fiber.

Let us be more precise. As in \cite{BH2}, we focus on the P.E.L cases of type A and C, allowing some ramification.
More precisely, we consider quasi-split unitary or symplectic groups over a ramified number field. In case C, we studied all 
cases in \cite{BH2}, proving that the model is smooth, and the ordinary locus is dense in the special fiber. In case 
A, we have a CM field $F$ over a totally real field $F_0$, and we studied in \cite{BH2} the geometry of the model and 
its special fiber at $p$ under the assumption that $F/F_0$ was unramified at $p$. In this case we showed also that 
the model is smooth and the $\mu$-ordinary locus is dense in the special fiber. Both of this results were expected 
since \cite{P-R}, and proved in some cases. It was clear, since the PhD thesis of Kramer (\cite{Kramer}) that the model 
could not be smooth if we allow $F/F_0$ to be ramified. In this article, we study the geometry of the special fiber of the 
Pappas-Rapoport model under this assumption (which is referred to as the (AR) case). We prove that the special fiber is stratified by an explicit poset with a combinatorial description, and in particular we have a description of the irreducible components of the special fiber. Moreover we prove the closure relations for this stratification. Note that the Rapoport locus coincides with one (maximal) stratum of the special fiber. Let us give a precise formulation when $F_0 = \QQ$, i.e. when $F$ is quadratic imaginary, and $p$ is a ramified prime, $\pi$ a uniformizer of $F_p := F \otimes_\QQ \QQ_p$. Let $a,b$ be integers with $a \leq b$, and let $Y$ be the Pappas-Rapoport model over $O_{F_p}$ (see Definition \ref{defin:PRmodel}) for the unitary group $GU(a,b)_{F/\QQ}$, and let $X$ be its special fiber. The entire point of the Pappas-Rapoport model in this case is that over $Y$, and thus over $X$, there is a locally direct factor $\omega_1 \subseteq \omega$ of rank $a$, where $\omega$ is the conormal sheaf of the universal abelian scheme. There is also a second locally direct factor $\omega_2 \subseteq \omega$ of rank $b$, which is obtained from $\omega_1$ and the polarization (see Definition \ref{defin:omega2}). Moreover the universal abelian scheme has an action of $O_F$, thus so does $\omega$.

For every $0 \leq h \leq \ell \leq a$, set
\[ X_{h,\ell} := \{ x \in X | \dim \pi \omega = h, \dim \omega_1 \cap \omega_2 = \ell\}.\]
This is a locally closed subscheme of $X$. For example $X_{a,a}$ is the (generalised) Rapoport locus. Our first result is the following,

\begin{theor}
Assume $p \neq 2$. For all $h \leq \ell$, the stratum $X_{h,\ell}$ is non empty, smooth, and equidimensional of dimension $ab - \frac{(\ell - h)(\ell - h + 1)}{2}$. Moreover we have the closure relations,
\[ \overline{X_{h,\ell}} = \coprod_{0 \leq h' \leq h \leq \ell \leq \ell' \leq a} X_{h',\ell'}.\]
In particular $X$ is not smooth, and the smooth locus is the union of the $X_{h,h}$ for $0 \leq h \leq a$. Moreover $Y$ is flat over $O_F$, normal, and $X$ is Cohen-Macaulay.
\end{theor}

When $F_0 \neq \QQ$, one has a similar description (the index of the stratification being more complicated), all computations reducing to the previous case. Let us stress that in the case of an unramified prime there is only one open stratum (the Rapoport locus) as in the cases considered in \cite{BH2}. In particular, in the situation considered here (i.e. the (AR) case), there is no chance that the Rapoport locus or the $\mu$-ordinary locus is Zariski dense. 
If $p=2$, we have partial similar results, which depend on the class of some pairing. In particular it can happen that some of the strata are empty (see Proposition \ref{prop:6.5}), and in general the smooth locus is more complicated than the union of the open strata (see Section \ref{sect:6}). 

We can then study how the stratification studied previously interacts with "classical" stratification, for example the one induced by the \textit{partial} Hasse invariants. The general result is likely to be overly complicated for combinatorial reasons, so let us describe the situation when $F_0 = \QQ, p \neq 2$ and $(a,b) = (1,n)$, $n \geq 1$. In this case the previous stratification gives only 3 strata, $R = X_{1,1}$, the Rapoport locus, $B = X_{0,0}$ the other open strata and the intersection of their closures $P = X_{0,1}$. We have two partial Hasse invariants $\hasse_1,\hasse_2$ (see Definition \ref{defin:hasse}) and we stratify further these 3 strata depending on the vanishing of these invariants. It turns out that there are restrictions on these vanishings, and we end up with 6 strata, refining the previous stratification : 
\[ R = R_0 \sqcup R_1 \sqcup R_2, \quad B = B_0 \sqcup B_1 \sqcup B_2, \quad P = P_0 \sqcup P_1 \sqcup P_2,\]
with $R_0 = X^{ord}$ the $\mu$-ordinary locus, $R_0,P_0,B_0$ the locus of non vanishing of both the partial Hasse invariants, and $R_2,P_2$ the locus of vanishing of both partial Hasse invariants (see Section \ref{sect:Hasse1n}). We then have the following description, maybe surprising for $\overline{B_1}$ and when $n=1$. 

\begin{theor}
If $n = 1,2$ the strata $R_1,P_1$ are empty. If $n = 1$ we have,
\[ \overline{X^{ord}} = X^{ord} \cup P_0, \quad \overline{R_2} = R_2 \cup P_2, \quad \overline{B_0} = B_0 \cup B_1 \cup B_2 \cup P_0 \cup P_2,\]
while $X^{ord},R_2,B_0$ are open, and $P_0,P_2,B_1,B_2$ are closed. \\
If $n \geq 2$, then $B_2$ and $P_2$ are closed, and we have the closure relations
\[\overline{X^{ord}} = X^{ord} \cup_{i=1}^2 R_i \cup_{i=0}^2 P_i \qquad  \overline{R_2} = R_2 \cup P_2\]
\[\overline{B_0} = \cup_{i=0}^2 B_i \cup_{i = 0}^2 P_i \qquad \overline{B_1} = B_1 \cup P_1 \cup P_2 \qquad \overline{P_0} = \cup_{i=0}^2 P_i.\]
If $n \geq 3$, one has moreover
\[\overline{R_1} = \cup_{i=1}^2 R_i \cup_{i=1}^2 P_i \qquad \overline{P_1} = P_1 \cup P_2.\]
\end{theor}

We expect that this combinatorial description of the special fiber will relate to more classical geometric varieties, and hopefully that we will be able to prove some cohomological vanishing of modular forms using the geometry of the model. As a first step, we can already prove that the extra irreducible components in the special fiber, i.e. those which are disjoint from the (generalised) Rapoport locus, do not contribute to modulo $p$ modular forms in sufficiently regular weights. Namely assume $F_0 = \QQ$ and let $\kappa = (k_1 \geq \dots \geq k_a,\ell_1 \geq \dots \geq \ell_b) \in \ZZ^{a+b}$ be a weight (see section \ref{section:modforms}). Then we have the following result.

\begin{theor}
If $h < a$ and if we cannot find $\{i_1 < \dots < i_{a-h}\} \subset \{1,\dots,a\}$ such that 
\[ k_{i_1} = \dots = k_{i_{a-h}} \leq \ell_{b-h+1},\]
then $H^0(\overline{X_{h,h}},\omega^\kappa) = 0$.

\end{theor}

We hope to generalise this result to higher cohomology and less restrictive weights.
In \cite{SYZ}, similar results on the geometry of the mod $p$ fibers of Shimura varieties are proven for the Pappas-Zhu model and its EKOR stratification. It would be interesting to know if our results are related to theirs. More recently, Zachos \cite{Zac22} has studied a similar problem under the assumption $F_0 = \QQ$ restricting to the geometry around points of $X_{0,a}$, but has given an explicite blow-up of the splitting model which has semi-stable reduction, under the extra assumption that $(a,b) = (2,n-2)$. \\

We would like to thank Fabrizio Andreatta for suggesting to have a look at \cite{SYZ}. We would like to heartily thank Ioannis Zachos for pointing a mistake in a first version of this paper. The authors are part of the project ANR-19-CE40-0015 COLOSS.

\section{Case of a quadratic imaginary $F$}

\subsection{Definition of the variety}

Let $F$ be a complex quadratic extension of $\QQ$, and assume that $p$ is ramified in $F$. Let $F_p$ be the completion of $F$ at $p$, and $\pi$ a uniformizer of $F_p$. We write $\sigma_1, \sigma_2$ the embeddings of $F_p$ into $\overline{\QQ_p}$, and let us define $\pi_i = \sigma_i(\pi)$. Let $a,b$ be integers with $a \leq b$, and define $m=a+b$.

\begin{defin}
\label{defin:PRmodel}
Let $Y$ be the moduli space over $O_{F_p}$ whose $R$-points are couples $(A, \lambda, \iota, \eta, \omega_1)$, where
\begin{itemize}
\item $A$ is an abelian scheme over $R$ of dimension $m$,
\item $\lambda$ is a polarization, principal at $p$,
\item $\iota : O_F \to End(A)$, making the Rosati involution and the complex conjugation compatible,
\item $\eta$ is a level structure away from $p$,
\item $\omega_1 \subseteq \omega_A$ is a locally direct factor of rank $a$, stable by $O_F$,
\item $O_F$ acts by $\sigma_1$ on $\omega_1$, and by $\sigma_2$ on $\omega_A / \omega_1$.
\end{itemize}
\end{defin}

Let $\mathcal{E} = H^1_{dR} (A)$; it is a locally free sheaf on $Y$ of rank $2m$. If has an action of $O_F$, and is locally free of rank $m$ over $O_Y \otimes_{\ZZ} O_F$. The Hodge filtration is $\omega_A \subseteq \mathcal{E}$. The sheaf $\mathcal{E}$ has an action of $O_{F_p}$, and let $[a]$ be the action of $a$ on $\mathcal{E}$ for every $a \in O_{F_p}$. The last condition implies that $([\pi] - \pi_1) \omega_1 = 0$ and $([\pi] - \pi_2) \omega \subseteq \omega_1$. \\ 
$ $\\
Thanks to the polarization, one has a perfect pairing on $<,>$ on $\mathcal{E}$. The condition between the Rosati involution and the complex conjugation implies that for all $x,y \in \mathcal{E}$ one has
$$<[a] \cdot x, y> = <x, [\overline{a}] \cdot y>$$
The Hodge filtration is totally isotropic for this pairing. Moreover, the above relation implies that 
$$\mathcal{E} [[\pi] - \pi_i]^\bot = \mathcal{E} [[\pi] - \pi_i]$$
where $\mathcal{E} [[\pi] - \pi_i]$ consists of the elements of $\mathcal{E}$ killed by $[\pi] - \pi_i$.

\begin{rema}
Since $\mathcal{E} [[\pi] - \pi_1]^\bot = \mathcal{E} [[\pi] - \pi_1]$, one has a perfect pairing between $\mathcal{E} [[\pi] - \pi_1]$ and $\mathcal{E} / \mathcal{E} [[\pi] - \pi_1]$. This last sheaf is isomorphic to $\mathcal{E} [[\pi] - \pi_2]$ via the multiplication by $[\pi] - \pi_1$. One has thus an induced pairing between $\mathcal{E} [[\pi] - \pi_1]$ and $\mathcal{E} [[\pi] - \pi_2]$, given by the formula
$$\{ ([\pi] - \pi_2)x , ([\pi] - \pi_1)y \} := <([\pi] - \pi_2)x, y>$$
\end{rema}

\begin{defin}
\label{defin:omega2}
Let us define $\omega_2 \subseteq \mathcal{E}$ by the formula 
$$\omega_2 = (([\pi] - \pi_2)^{-1} \omega_1   )^\bot$$
\end{defin}

\begin{prop}
The sheaf $\omega_2$ is locally free of rank $b$, and one has $\omega_2 \subseteq \omega$. Moreover, one has 
$$([\pi] - \pi_2) \cdot \omega_2 = 0 \qquad ([\pi] - \pi_1) \cdot \omega  \subseteq \omega_2$$
\end{prop}

\begin{proof}
From the properties satisfied by $\omega_1$, one has $\omega \subseteq ([\pi] - \pi_2)^{-1} \omega_1$. Taking the orthogonal of this relation (ans using that $\omega^\perp = \omega$), one finds the relation $\omega_2 \subseteq \omega$. \\
One has $\mathcal{E} [[\pi] - \pi_2] \subseteq \omega_2^\perp$, and taking the orthogonal gives $\omega_2 \subseteq \mathcal{E} [[\pi] - \pi_2]$. In other words, $([\pi] - \pi_2) \cdot \omega_2 = 0 $. \\
For the last point, we first claim that $([\pi] - \pi_1) \cdot \omega_1^\perp = \omega_2$. Indeed, let $x \in \omega_2$; since it belongs to $\mathcal{E} [[\pi] - \pi_2]$ there exists $x' \in \mathcal{E}$ such that $x= ([\pi] - \pi_1) x'$. Then
\begin{equation*} 
\begin{split}
x \in \omega_2 & \Leftrightarrow \quad <x,y>=0  \quad \forall y \in ([\pi] - \pi_2)^{-1} \omega_1 \\ 
& \Leftrightarrow \quad <x', ([\pi] - \pi_2)y > = 0 \quad \forall y \in ([\pi] - \pi_2)^{-1} \omega_1  \\
& \Leftrightarrow \quad x' \in \omega_1^\perp
\end{split}
\end{equation*} 

\noindent One thus has $([\pi] - \pi_1) \cdot \omega_1^\perp = \omega_2$, or equivalently $\omega_1^\perp = ([\pi] - \pi_1)^{-1} \omega_2$. The inclusion $\omega_1 \subseteq \omega$ then implies that $\omega \subseteq ([\pi] - \pi_1)^{-1} \omega_2$. In other words, $([\pi] - \pi_1) \cdot \omega  \subseteq \omega_2$.
\end{proof}

\subsection{Geometry of the special fiber}

Let $X$ be the special fiber of $Y$. Over $X$, the sheaf $\mathcal{E}$ is locally free of rank $m$ over $O_{X} [\pi] / \pi^2$. The sheaves $\omega_1, \omega_2$ are in $\mathcal{E} [\pi]$, and contain $\pi \cdot \omega$.

\begin{rema}
The sheaf $\mathcal{E} [\pi]$ is totally isotropic, but is endowed with a perfect modified pairing given by
$$\{ \pi x, \pi y \} := <\pi x,y>$$
This pairing is \textbf{symmetric}; indeed since $\overline{\pi} = - \pi$ in the residue field of $F$, one has
$$\{\pi y, \pi x\} = <\pi y , x> = <y, \overline{\pi} x> = <\pi x , y> = \{\pi x, \pi y \}.$$
If we want to denote the orthogonal of a subspace $\mathcal F \subset \mathcal E[\pi]$ for this new pairing, we denote it by $\mathcal F^{\perp'}$ to highlight the difference with the usual pairing $<,>$, where we use the notation $\mathcal F^{\perp}$.
\end{rema}

\begin{defin}
Let $k$ be a field in characteristic $p$, and let $x \in X(k)$. Let us define the integers $(h(x), l(x))$ as the dimension of $\pi \cdot \omega$, and $\omega_1 \cap \omega_2$ respectively.
\end{defin}

\begin{rema}
From the previous section, one gets that $\omega_2$ is the orthogonal of $\omega_1$ in $\mathcal{E} [\pi]$, for the modified pairing.
\end{rema}

\begin{prop}
Let $k$ be a field of characteristic $p$, and let $x \in X(k)$. Then one has 
$$0 \leq h(x) \leq l(x) \leq a$$
\end{prop}

The integers $h(x), l(x)$ will allow us to define a stratification on $\overline{X}$. Indeed, one has 
$$X = \coprod_{0 \leq h \leq l \leq a} X_{h,l} $$
where $X_{h,l}$ consists in the points $x$ with $(h(x), l(x)) = (h,l)$. 

\begin{prop}
Let $(h,l)$ be integers with $0 \leq h \leq l \leq a$, and let  $\overline{X_{h,l}}$ be the closure of $X_{h,l}$. Then 
$$\overline{X_{h,l}} \subseteq \coprod_{0 \leq h' \leq h \leq l \leq l' \leq a} X_{h',l'}$$
\end{prop}

\begin{proof}
The integer $h$ is equal to the dimension of $\pi \cdot \omega$. It thus decreases by specialization. The integer $l$ is equal to the dimension of $\omega_1 \cap \omega_2$. This quantity increases by specialization.
\end{proof}

\noindent In particular, the stratum $X_{0,a}$ is closed and the strata $X_{h,h}$ are open, for every $0 \leq h \leq a$.

\subsection{A remark on deformations of a $p$-divisible group}

Let $G$ be a $p$-divisible group over $k$, a field of characteristic $p$. Let $R = k[[t]]$ and $R_n = k[t]/(t^n)$, with the obvious maps. Let $\mathbb D$ the crystal of $G$ and $\mathcal E = \mathbb D_{k \fleche k}$. There are no divided powers on $k[[t]] \fleche k$, thus we can't a priori evaluate $\mathbb D$ on $k[[t]]$. Let $\widetilde{\mathcal E} = \mathcal E\otimes_k k[[t]]$. We denote $\omega \subset \mathcal E$ the Hodge filtration of $G$ (with extra structure).

\begin{prop}
Let $\widetilde{\omega} \subset \widetilde{\mathcal E}$ be a locally direct factor lifting $\omega \subset \mathcal E$. Then there exists a $p$-divisible group $\widetilde G$ over $k[[t]]$ lifting $G$ (with extra-structure), such that, when evaluating $\mathbb D(\widetilde G)_{k[[t]] \fleche k[[t]]} = \widetilde{\mathcal E}$, and the Hodge filtration is given by $\widetilde\omega$.
\end{prop}

\begin{proof}
All displays are Dieudonne displays (which we can consider since our ring is $R_n = k[t]/(t^n)$ or $R = k[[t]]$). We also denote $W(R)$ instead of $\hat{W}(R)$. Let $P$ be the Display of $G$, and set $\widetilde P = P \otimes_{W(k)} W(R)$ and $P_n = P\otimes_{W(k)} W(k[t]/(t^n)) = \widetilde P \otimes_{W(k[[t]])} W(k[t]/(t^n))$.
The map $k[[t]]/(t^n) \fleche k[[t]]/(t^{n-1})$ is endowed with (nilpotent) divided powers. In particular, for $n=2$, we get from $\widetilde\omega \otimes_{k[[t]]} k[t]/(t^2) \subset \widetilde E \otimes_{k[[t]]} k[t]/(t^2)$ a lift of the Hodge filtration which induces the existence of a display $\mathcal P_2$ with $W(R_2)$-module $P_2$, itself corresponding to a $p$-divisible group $G_2$ over $R_2$ lifting $G$, with $\mathbb D(G_2)_{R_2 \fleche R_2} = \mathbb D(G)_{R_2 \fleche k} = P_2$ and corresponding Hodge filtration (cf \cite{Zink2} Theorem 4,\cite{Mess}).
Assume the corresponding result at rank $n$. In particular we have a display $\mathcal P_n$, with module $P_n$ and Hodge filtration corresponding to 
$\widetilde w \otimes_R R_n \subset P_n/I_{R_n}P_n = \widetilde{\mathcal E} \otimes_R R_n$. As the map corresponding to $R_{n+1} \fleche R_n$ has divided powers 
and we have a $R_{n+1}$-triple (by base change), the lift of the Hodge filtration 
$\widetilde w \otimes R_{n+1} \subset \widetilde{\mathcal E} \otimes R_{n+1} = P_{n+1}/I_{R_{n+1}}P_{n+1}$ induces a lift $\mathcal P_{n+1}$ of $\mathcal P_n$, 
with Hodge filtration determined by $\widetilde \omega$. To $\mathcal P_{n+1}$ by \cite{Zink2} Theorem 20 we have an associated $p$-divisible group $G_{n+1}$ over $R_{n+1}$. 
Moreover, by \cite{LauRel} Theorem B, we have 
$\mathbb D(G_{n+1})_{k[t]/(t^{n+1}) \fleche k[t]/(t^{n+1})} = \mathbb D(G)_{k[[t]]/(t^n)\fleche k[[t]]/(t^n)} = \mathbb D(\mathcal P_{n+1})_{k[[t]]/(t^n)} =
 \mathcal P_{n+1}/I_{R_{n+1}}P_{n+1}= \widetilde{\mathcal E} \otimes_R R_{n+1}$ with compatibility with the Hodge filtration. We then set $G = \varprojlim G_n$, a $p$-divisible group over $k[[t]]$, satisfying the desired assumptions.
\end{proof} 

Using this proposition, by abuse we call $\widetilde{\mathcal E}$ the evaluation of $\mathcal D(G)$ on $k[[t]]$.

\subsection{Closure relations and geometry of strata}

In this section, we assume that $p \neq 2$. Let us prove a first proposition about the closure relations for the strata.

\begin{prop}\label{propclosurestrata}
Let $(h,l)$ be integers with $0 \leq h \leq l \leq a$. One has
\[\overline{X_{h,l}} = \coprod_{0 \leq h' \leq h \leq l \leq l' \leq a} X_{h',l'}.\]
\end{prop}

\begin{proof}
The previous proposition gives the expected inclusion, we will now show the converse. Let us first prove that $X_{0,a}$ is in the closure of $X_{h,l}$ for every $0 \leq h \leq l \leq a$. Let $k$ be an algebraically closed field of characteristic $p$, and let $x \in X_{0,a} (k)$. We will prove that for any $h \leq l$, one can find a generization of $x$ which lies in $X_{h,l}$. \\
Since $(h(x), l(x)) =(0,a)$, one has the inclusions $\omega_1 \subseteq \omega_2 \subseteq \omega = \mathcal{E} [\pi]$. The matrix of the modified pairing on this set, with some appropriate basis, is
\begin{equation}\label{eq:pairing}\left ( \begin{array} {ccc}
0 & 0 & I_a \\
0 & I_{b-a} & 0 \\
I_a & 0 & 0
\end{array} \right )\end{equation}
Let $\widetilde{\mathcal{E}}$ be the crystal evaluated at $k[[t]]$. We will investigate lifts of the modules $\omega_1 \subseteq \omega$ to $\widetilde{\mathcal{E}}$. First, a lift $\widetilde{\omega_1}$ of $\omega_1$ inside $\widetilde{\mathcal{E}} [\pi]$ is given by a matrix $\left( \begin{array}{c} 
I_a \\
X \\
Y 
\end{array} \right)$ where $X,Y$ are matrices with coefficients in $k[[t]]$, whose reductions are $0$ modulo $t$. One computes that the orthogonal of $\widetilde{\omega_1}$ inside $\widetilde{\mathcal{E}} [\pi]$ is given by $\left( \begin{array}{cc} 
I_a & 0 \\
0 & I_{b-a} \\
- ^t Y & -^t X 
\end{array} \right)$. A lift $\widetilde{\omega}$ of $\omega$ will thus contain the vectors $\left( \begin{array}{c} 
0 \\
I_{b-a} \\
- ^t X
\end{array} \right)$, and be contained in $\pi^{-1} \widetilde{\omega_1}$. Let us call $\pi \widetilde{e_1},\dots,\pi \widetilde{e_a}$ the vectors appearing in the matrix for $\widetilde{\omega_1}$, and choose $\widetilde{e_1},\dots,\widetilde{e_a}$ some preimage by $\pi$ in $\widetilde{\mathcal E}$. Up to change $\widetilde{e_i}$ by some $\pi$-torsion element we can assume that they are two by two orthogonal for $<,>$.
Considering the vectors $\pi e_{b+1}, \dots, \pi e_{a+b}, \widetilde{e_1}, \dots, \widetilde{e_a}$, the remaining vectors for $\widetilde \omega$ are given by a matrix $\left( \begin{array}{c} 
I_a \\
Z
\end{array} \right)$ in the previous family, where $Z$ is a matrix with coefficients in $k[[t]]$ whose reduction is $0$. The condition that $\widetilde{\omega}$ is totally isotropic is equivalent to the equations 
\[ Z = ^t Z \qquad (Y + ^t Y + ^t X X ) Z =0.\]
Indeed, we have that the image by $\pi$ of the element corresponding to $\left( \begin{array}{c} 
I_a \\
Z
\end{array} \right)$ are given by $Z$ in the basis $\pi \widetilde e_1,\dots \pi \widetilde e_a$ i.e. in the original basis $\pi e_1,\dots, \pi e_n$ by 
$\left( \begin{array}{c} 
Z \\
X{^t}Z\\
Y{^t}Z
\end{array} \right)$. The computation of $<w_1,w_2>$ for $w_1,w_2 \in Z(\widetilde e_1,\dots  \widetilde e_a) + \pi( e_{b+1},\dots,e_{h})$ (written symbolically) is given by
\begin{eqnarray*} <w_1,w_2> = <Z \sum \lambda_i \tilde e_i + \sum \lambda_i \pi e_{b+i}, Z \sum \mu_i \tilde e_i + \sum \mu_i \pi e_{b+i})> \\ = <Z \sum \lambda_i \tilde e_i, Z \sum \mu_i \tilde e_i > + <Z \sum \lambda_i \tilde e_i,  \sum \mu_i \pi e_{b+i}> + <\sum \lambda_i \pi e_{b+i}, Z \sum \mu_i \tilde e_i >,\end{eqnarray*}
as $\mathcal E[\pi]$ is its own orthogonal, and the first term is zero as the $\widetilde e_i$ are two by two isotropic. Thus we are left with
\[ <w_1,w_2> = -\{ Z \sum \lambda_i \pi \tilde e_i,  \sum \mu_i \pi e_{b+i}\} + \{\sum \lambda_i \pi e_{b+i}, Z \sum \mu_i \pi \tilde e_i \},\]
which is given in matrix terms, varying $w_1,w_2$ and using the shape of the divided pairing by $-{^t}Z + Z = 0$. The second equation is similar using orthogonality between $\widetilde w_1$ and the vector corresponding to $\left( \begin{array}{c} 
Z \\
X{}Z\\
Y{}Z
\end{array} \right)$. The last equation is actually automatic.
From Grothendieck-Messing, applied to a devissage to square zero ideals corresponding to $k[T]/(T^n) \fleche k[T]/(T^{n-1})$, and Serre-Tate, this deformation of the Hodge filtration gives a generization $\widetilde{x}$ of $x$. We will use this argument repetitively. One then sees that $h(\widetilde{x})$ is equal to the rank of $Z$, and $l(\widetilde{x})$ is equal to the nullity of the matrix $Y + ^t Y + ^t X X $. For any couple $(l,h)$ with $0 \leq h \leq l \leq a$, one can find matrices $X,Y,Z$ such that the rank of $Z$ is $h$, the rank of $Y + ^t Y + ^t X X$ is $a-l$ and the above equations are satisfied, hence the result. \\
The general case can be deduced by a similar discussion : choose $h < \ell$ and $e_1,\dots,e_{a+b}$ a basis of $\mathcal E[\pi]$ such that $e_1,\dots,e_h$ is a basis of $\pi \omega$, $e_1,\dots,e_\ell$ a basis of $\omega^1\cap \omega^2$, $e_1,\dots,e_a$ is a basis of $\omega^1$, $e_1,\dots,e_\ell,e_{a+1},\dots,e_{b+a-\ell}$ a basis of $\omega^2$, and $e_1,\dots,e_{a+b-h}$ is a basis of $\omega[\pi]$. We moreover assume that the divided pairing is given by the matrix (in this basis)
\[
\left(
\begin{array}{cccccc}
  &   &   &   &  & I_h \\
  &   &   &   & I_{\ell-h} &  \\
  &   &  I_{a-\ell}  & 0  & &  \\
  &   & 0    & I_{b-\ell}  & & \\
    & I_{\ell-h}  &    &   & &  \\
I_h  &   &     &   & &
\end{array}
\right)
\]
We then set in $\widetilde{\mathcal E}[\pi]$ a lift $\widetilde \omega_0$ of $\pi \omega$, $\widetilde \omega_{int} \supset \widetilde\omega_0$ of $\omega^1\cap\omega^2$, and $\widetilde{\omega_1} \subset \widetilde \omega_{int}$ of $\omega_1$ by the column of the matrix
\[
\left(
\begin{array}{ccc}
 I_h&   &   \\
  & I_{\ell-h}  &    \\
  & Y_2  &  I_{a-\ell}    \\
  &   &     \\
    &0  &      \\
  &   &     
\end{array}
\right)
\]
for some matrix $Y_2 \in M_{a-\ell,\ell-h}(tk[[t]])$. Denote $\tilde e_1,\dots,\tilde e_h = e_1,\dots,e_h$ the first vectors in the previous matrix, which gives a basis for $\widetilde \omega_0$ and denote $\tilde e_{h+1},\dots,\tilde e_{\ell}$ the next $\ell-h$ ones. We then check that the orthogonal of $\widetilde \omega_1$ for the divided pairing in $\widetilde E[\pi]$ is given by \[
\left(
\begin{array}{ccc}
 I_h&   &   \\
  & I_{\ell-h}  &    \\
  &   & 0   \\
  &   &    I_{b-\ell} \\
    &0  &      \\
  &   &     
\end{array}
\right)
\]
and this defines a lift $\widetilde \omega_2$ of $\omega_2$ and, after inverting $t$, the intersection of $\widetilde \omega_2 \cap \widetilde \omega_1$  
contains $\widetilde \omega_0$ and a space related to the kernel of the matrix $Y_2$. Namely, 
$\dim \widetilde \omega_2 \cap \widetilde \omega_1[1/t] = h + \dim \ker Y_2$. Finally we lift $\omega$ by adding to $\widetilde \omega_1$ the vectors
\[\left(
\begin{array}{c}
   0 \\
  0    \\
 0   \\
   I_{b-\ell} \\
    0  \\
     0
\end{array}
\right)
\]
of $\widetilde E[\pi]$ together with the following vectors : choose $\pi^{-1}\tilde e_1,\dots,\pi^{-1}\tilde e_d$ a family of vectors of $\omega$ which generates $\omega[\pi]$ after multiplying by $\pi$, and choose lifts in $\widetilde E$ which moreover maps to $\tilde e_1,\dots,\tilde e_h = e_1,\dots,e_h$ (the previous basis for $\widetilde \omega_0$) and choose $\pi^{-1}\tilde e_{h+1},\dots,\pi^{-1}\tilde e_{\ell}$ in the preimage by $\pi$ of $\tilde e_{h+1},\dots,\tilde e_\ell$ such that $\pi^{-1}\tilde e_1,\dots,\pi^{-1}\tilde e_\ell$ are two by two orthogonal for the original pairing $< , >$. Then we add $\pi^{-1}\tilde e_1,\dots,\pi^{-1}\tilde e_h$ and the vectors given in the basis $e_{b+a-\ell + 1},\dots,e_{b+a-d},\pi^{-1}\tilde e_{d+1},\dots,\pi^{-1}\tilde e_\ell$ by the matrix $\left(\begin{array}{c} I_{\ell-d} \\ Z \end{array}\right)$ i.e. in the original basis $\pi^{-1}e_1,\dots,\pi^{-1}(e_{a+b})$ by

\[
\left(
\begin{array}{cc}
I_h& 0   \\
0& Z     \\
 0&  Y_2Z     \\
0 &0   \\
0&    \pi I_{\ell-h}        \\
0&     0  
\end{array}
\right)
\]
Because of the assumption on $\pi^{-1}\tilde e_1,\dots,\pi^{-1}\tilde e_\ell$ these last $\ell$ vectors are two by two isotropic for the original pairing iff (this reduces to a calculation with the divided pairing $\{,\}$) $^{t}Z - Z = 0$. Moreover $\widetilde \omega$ is totally isotropic (for $<,>$) if moreover $Y_2 Z = 0$. But clearly the rank, after inverting $t$, of $\pi \widetilde \omega$ is $h + \rk Z$, and thus if we have $h \leq h + r \leq \ell' = h+s \leq \ell$ for some $s,r \geq 0$, with $\ell - h \geq s \geq r$ then we can choose a symmetric matrix $Z$ of rank $r$ and a matrix $Y_2$  with kernel of dimension $s \geq r$ such that $Y_2 Z = 0$. Concluding as in the case of $X_{0,a}$, we have the result.
 \end{proof}

\begin{prop}
For all $h \leq \ell$, the stratum $X_{h,l}$ is nonempty {and smooth}, and is equidimensional of dimension $ab - \frac{(l-h)(l-h+1)}{2}$.
\end{prop}

\begin{proof}
To prove that all the strata are non empty, by the proposition above, it is enough to prove that $X_{0,a}$ is non empty. Let $k = \overline{\mathbb{F}_p}$, and let $E$ be an elliptic curve over $k$ with complex multiplication by $O_F$ such that the action of $O_F$ on $\omega_E$ is given by $\sigma_1$, and let $E^c$ be the same elliptic curve, but with the action of $O_F$ twisted by the complex conjugation. Define $A = E^a \times (E^c)^b$, and let us choose a space $\omega_1 \subseteq \omega_A$ which is totally isotropic for the modified pairing. This gives a point in $X_{0,a}$.
Let us now compute the dimension of the stratum $X_{h,l}$. On this stratum, on has the sheaves
$$\pi \omega \subseteq \omega_1 \cap \omega_2 \subseteq \omega_1$$
which are locally free of rank $h,l,a$ respectively. Deforming a point of $X_{h,l}$ inside $X_{h,l}$ thus consists in the following operations:
\begin{itemize}
\item deforming the sheaf $\pi \omega$ in a sheaf $\widetilde{\omega_0}$, which should be totally isotropic for the modified pairing.
\item deforming the sheaf $\omega_1 \cap \omega_2$ inside the orthogonal of the previous one, which should also be totally isotropic (for the modified pairing).
\item deforming the sheaf $\omega_1$ inside the orthogonal of the previous one (for the modified pairing), {asking moreover that $\omega_1/(\omega_1\cap\omega_2) \cap \left(\omega_1/(\omega_1\cap\omega_2)\right)^{\perp''} = \{0\}$, where $\perp''$ is the modified pairing descended to $(\omega_1\cap\omega_2)^{\perp'}/(\omega_1\cap\omega_2)$.}
\item deforming $\omega$, which should contain the orthogonal of $\widetilde{\omega_0}$ (for the modified pairing)and be contained in $\pi^{-1} \widetilde{\omega_0}$, and be totally isotropic (for the original pairing). Note that\footnote{$\pi \mathcal F^\perp = \mathcal F^{\perp'}$ if $\mathcal F \subset \mathcal E[\pi]$ and $\mathcal G^\perp = (\pi \mathcal G)^{\perp'}$ if $\mathcal E[\pi] \subset \mathcal G$} if we denote $\mathcal F = (\widetilde \omega_0)^{\perp'}$ we have $\mathcal F \subset \mathcal E[\pi]$ thus $\pi \mathcal F^\perp = \mathcal F^{\perp'} = \widetilde \omega_0$. But both $\mathcal F^\perp$ and $\pi^{-1} \widetilde w_0$ contains $\mathcal E[\pi]$, and are equal after multiplying by $\pi$, thus $\mathcal F^\perp = \pi^{-1}\widetilde \omega_0$. It is thus enough to deform the image of $\omega$ in $\mathcal F^\perp/\mathcal F$.
\end{itemize}
{In other words we look at the following sequence of schemes}
\[ \Gr^{Sp}(h,(\mathcal F^{\perp}/\mathcal F,<,>)) \fleche  U \fleche \Gr^O(l - h,((\omega_0)^{univ,\perp'}/\omega_0^{univ},\{,\})) \fleche \Gr^O(h,(\mathcal E[\pi],\{,\})),\]
where $\Gr^O(k,(V,\{,\}))$ is the Grassmanian of totally isotropic subspace of rank k in a space $V$ with symmetric pairing, and $\Gr^{Sp}(k,(V,<,>)))$ is the analogous one for an alternated pairing, $(\omega_0)^{univ}$ is the universal object of $ \Gr^O(h,((\mathcal E[\pi],\{,\}))$, and $U \subset \Gr^{O}(a-l,(\omega_1\cap\omega_2)^{univ,\perp'}/(\omega_1\cap\omega_2^{univ}),\{,\})$ is the open where the universal isotropic subspace $F$, which corresponds to $\omega_1/(\omega_1\cap\omega_2)$, satisfies $F^{\perp'} \cap F = \{0\}$, with $\omega_1\cap\omega_2^{univ}$ the (pullback of the) universal object of  $\Gr^O(l - h,((\omega_0)^{univ},\{,\}))$. 

All those Grassmanians are relatively smooth, and
the first point gives a dimension $h(a+b-h) - \frac{h(h+1)}{2}$, the second one a relative dimension $(l-h)(a+b-l-h) - \frac{(l-h)(l-h+1)}{2}$, the third one $(a-l)(b-l)$ and the last one $\frac{h(h+1)}{2}$. The total dimension is then
$$ab - \frac{(l-h)(l-h+1)}{2}.$$
{Moreover if $x \in X_{h,l}(S)$ is a $R$-point for some ring $R$, and $S\fleche R$ is a square-zero thickening of $\mathbb F_p$-schemes, then to lift $x$ it is enough to lift its Hodge Filtration (by Grothendieck Messing) satisfying all the desired properties. But this is indeed possible as all the previous Grassmanians are formally smooth and the dimension is indeed the given one.}

\end{proof}

\begin{cor}
The irreducible components of the scheme $X$ are determined by the ones of $X_{h,h}$ for $0\leq h \leq a$.
\end{cor}

\subsection{Local rings}

Assume $p \neq 2$.

\begin{prop}
The smooth locus of $X$ is the union of the strata $X_{h,h}$ for $0 \leq h \leq a$.
\end{prop}

\begin{proof}
We have seen that the strata $X_{h,h}$ are open and smooth in the previous proposition, thus their union is included in the smooth locus.
Assume now that $h < l$, and let $x \in X_{h,l} (k)$. We will prove that $X$ is not smooth at $x$. The pairing on $\mathcal{E} [\pi]$ is given by a matrix
$$\left ( \begin{array} {ccc}
0 & 0 & I_l \\
0 & I_{a+b-2l} & 0 \\
I_l & 0 & 0
\end{array} \right )$$
written with respect to a basis $\pi e_1, \dots, \pi e_{a+b}$, where $\omega_1$ is spanned by $\pi e_1, \dots, \pi e_a$, $\omega_1 \cap \omega_2$ is spanned by $\pi e_1, \dots, \pi e_l$ and $\pi \omega$ is spanned by $\pi e_{l-h+1}, \dots, \pi e_l$. One can also assume that $\omega/\omega_1$ is spanned by $\pi e_{a+1}, \dots, \pi e_{a+b-h} , e_{l-h+1}, \dots, e_l$. Let $\mathcal{E}'$ be the evaluation of the crystal at $k[\varepsilon] / \varepsilon^2$. Let us define a lift $\omega'$ of $\omega$ to $\mathcal{E}'$. One can lift the basis on $\mathcal{E}$ to a basis $e_1', \dots, e_{a+b}'$ in such a way that the matrix of the pairing is not changed. We define $\omega_1'$ to be spanned by $\pi e_1'+\varepsilon \pi e_{a+b-l+1}', \pi e_2', \dots, \pi e_a' $. We thus define $\omega'/\omega_1'$ to be spanned by $\pi e_{a+1}', \dots, \pi e_{a+b-l+1}' + \varepsilon e_1' , \dots, \pi e_{a+b-h}' , e_{l-h+1}', \dots, e_l'$. This gives a point $x' \in X(k[\varepsilon]/\varepsilon^2)$. We will now prove that this point cannot be lifted to $k[\varepsilon]/\varepsilon^3$. If it were the case, one would have a lift $\widetilde{\mathcal E}$ of $\mathcal{E}'$ to $k[\varepsilon]/\varepsilon^3$ together with a lift $\widetilde{\omega}$ of $\omega'$. In particular, there would exist an element in $\widetilde{\omega_1}$ of the form $v_1 := \pi \widetilde{e_1} + \varepsilon \pi \widetilde{e_{a+b-l+1}} + \varepsilon^2 \pi u$. There would also be in $\widetilde{\omega}$ an element of the form $v_2 := \pi \widetilde{e_{a+b-l+1}} + \varepsilon (\widetilde{e_1} + \varepsilon \widetilde{e_{a+b-l+1}}) + \varepsilon^2 v$ such that $\varepsilon^2 \pi v$ belongs to $\widetilde{\omega_1}$.
The (original) pairing between these two vectors is equal to
\[ <v_1,v_2> = \{ v_1,\pi v_2 \} = 2 \varepsilon^2 +  \{ \varepsilon^2\pi \widetilde{e_1} , \pi v \}.\] But the quantity $\{ \varepsilon^2\pi \widetilde{e_1} , \pi v \} =  0$ modulo $\varepsilon^3$, since $\varepsilon^2\pi v$ belongs to $\widetilde{\omega_1}$. This gives the desired contradiction, since $\widetilde{\omega} \ni v_1,v_2$ should be totally isotropic.    
\end{proof}

\begin{rema}
Even though the irreducible components of $X$ are determined by the $X_{h,h}$ (namely they are locally the closure of $X_{h,h}$ on the local model), and $X_{h,h}$ is smooth, it is not true in general that the irreducible components are smooth. This is true when $a = 1$ (\cite{Kramer}), but already for $a=2$ \cite{Zac22} computed that the irreducible component corresponding to $X_{1,1}$ is not smooth. Our computations prove that the irreducible component corresponding to $X_{h,h}$ is not smooth when $0 < h <a$. The variety $X$ is thus not semi-stable when $a \geq 2$.
\end{rema}
 
\begin{prop}
\label{prop:flat0}
The scheme $Y$ is flat over $\mathcal O_{F_p}$, normal and its special fiber is Cohen-Macaulay.
\end{prop}

\begin{proof}
The special fiber of $Y$, $X$, is reduced as its irreducible components are smooth. Thus we only need to prove that generic points of irreducible components of $X$ lifts to characteristic zero. We translate the formulation this way. Let $\Lambda = \mathcal O_F^{2m}$, with natural polarisation $(x,y) = \sum_i x_i c(y_i)$. We have a local model diagram
\[ Y \longleftarrow \widetilde{Y} \fleche \mathcal N,\]
where $\widetilde{Y} = \Isom_{\mathcal O_F,<,>}(\mathcal E,\Lambda \otimes O_Y)$, and $\mathcal N$ is the $\mathcal O_{F,p}$-scheme parametrizing, for any scheme $S$, triples
\[ (\mathcal F_1,\mathcal F),\]
where $\mathcal F \subset \Lambda\otimes_{O_F} S$ is a locally direct factor, stable by $\mathcal O_F$, totally isotropic of rank $a+b$, and $\mathcal F_1 \subset \mathcal F$ is a locally direct factor of rank $a$, stable by $\mathcal O_F$, such that
\begin{itemize}
\item $([\pi]-\pi_1)(\mathcal F_1) \subset p\mathcal F$
\item $([\pi]-\pi_2)(\mathcal F) \subset \mathcal F_1.$
\end{itemize}
where the map $\widetilde{Y} \fleche Y$ is the natural one and is smooth (it is a torsor over a smooth algebraic group), and \[\widetilde Y \fleche \mathcal N, (A/S,\omega^1,i : \mathcal E \simeq \Lambda \otimes \mathcal O_S) \mapsto \omega^1\subset \omega \subset \mathcal E \simeq\Lambda\otimes \mathcal O_S.\] 
In particular $\widetilde Y \fleche \mathcal N$ is formally smooth by Grothendieck-Messing. Thus to prove the result, it is enough to show that $\mathcal N$ is flat over $\mathcal O_{F,p}$.
But what we did before actually shows that, if $N = \mathcal N \otimes_{O_F} k_F$,
\[ N = \coprod_{0 \leq h \leq \ell \leq a} N_{h,l},\]
with the expected closure relation, and each $N_{h,l}$ is smooth of some dimension. In particular, $N$ is reduced and the irreducible components of $N$ are those of the various $N_{h,h}$. In particular, $N$ has dimension $ab$ and $\mathcal N$ has dimension $ab+1 = \dim N + \dim \mathcal O_F$. 
Thus to show that $\mathcal N$ is flat over $\mathcal O_F$ it is enough to show its generic fiber is dense. 
But $N$ is endowed with an action of 
\[ G = \{ g \in \GL_{a+b,k[X]/(X^2)}(\Lambda) |  <g x, y> = <x,gy>\},\]
and $G$ acts transitively on $X_{h,\ell}$ for each $h,\ell$ : Indeed, for each such we can find a basis $e_1,\dots,e_{a+b}$ of $\Lambda$ such that $e_1,\dots,e_h,Xe_1,\dots,X e_{a+b-h}$ is a basis of $\mathcal F$ and $X e_1,\dots,X e_a$ is a basis of $\mathcal F_1$, with $(F_1)\cap(F^1)^{\perp}$ given by $X e_1,\dots,Xe_\ell$. But $G$ preserve the non-flat (resp. non-Cohen-Macaulay) locus, so as the flat (resp. Cohen-Macaulay) locus is open, it is enough to show that points of $N_{0,a}$ lifts to characteristics 0 (and are Cohen-Macaulay in $N$).
\begin{claim} Let $x \in N_{0,a}(k)$. The local ring at $x$ of $\mathcal N$ is given by (the localisation at $Z=T=Y=\pi=0$ of)
\[ W_{\mathcal{O}} (k)[Z,X,Y]/(Z-{^t}Z,(Y+{^t}Y + {^t}XX)Z-(\pi_1-\pi_2)I_a).\]
where $W_{\mathcal{O}} (k) = W(k) \otimes_{\mathbb{Z}_p} \mathcal O_{F_p}$, $Z$ is a (symmetric) $a \times a$-matrix, $Y$ is a $a \times a$-matrix and $X$ is a $(b-a) \times a$-matrix.\end{claim} 

Assume the Claim is true for the moment. The previous ring is flat over $\mathcal O_{F_p}$, with Cohen-Macaulay (and reduced) fibers by a result of \cite[Theorem 2]{CN} and \cite[Remarque 5.9]{DP}. Indeed, it is proven in \cite{CN} that, for $k$ a field, the ring
\[ k[Z,T]/(Z-{^t}Z,T-{^t}T,ZT),\]
where $Z,T$ are two (symmetric) matrices of size $a$, is Cohen-Macaulay (and reduced) and the generic fiber is smooth. But if $\Char(k) \neq 2$, then the map 
\[
\begin{array}{ccc}
   W_\mathcal O(k)[Z,X,T,A]/((T+{^t}XX)Z-(\pi_1-\pi_2)I_a) &\fleche& W_\mathcal O(k)[Z,X,Y]((Y+{^t}Y + {^t}XX)Z-(\pi_1-\pi_2)I_a)  \\
  T &\mapsto& Y + {^t}Y\\ A &\mapsto&Y - {^t}Y,
  \end{array}
\]
where $Z,T$ are symmetric size $a$ matrix and $A$ is an antisymmetric size $a$ matrix, is an isomorphism. So 
\begin{eqnarray*} \Spec(W_\mathcal O(k)[Z,X,Y](Z-^{t}Z,(Y+{^t}Y + {^t}XX)Z-(\pi_1-\pi_2)I_a)) \\ \simeq
 \Spec(W_\mathcal O(k)[Z,X,T]/(Z-{^t}Z,T-{^t}T,(T+{^t}XX)Z-(\pi_1-\pi_2)I_a)) \times \mathbb 
A_{W_\mathcal O(k)}^{\frac{a(a-1)}{2}}.\end{eqnarray*}
Now, we have the  isomorphism
\[
   W_\mathcal O(k)[Z,X,T]/(Z-{^t}Z,T-{^t}T,(T+{^t}XX)Z-(\pi_1-\pi_2)I_a) \fleche W_\mathcal O(k)[Z,X,T]/(Z-{^t}Z,T-{^t}T,TZ-(\pi_1-\pi_2)I_a),\]
   given by
 \[ T \mapsto T -{^t}XX, X \mapsto X.\]
where the last ring is the one of \cite{DP} except that $(\pi_1-\pi_2)$ appears instead of $p$ and the special fiber is the same, so their argument that irreducible components in special fiber lifts carries over (changing $p$ by $\pi_1-\pi_2$).
Thus the ring of the Claim has reduced and Cohen-Macaulay fibers with every irreducible component being in characteristics zero, giving that the ring is flat. Thus 
$\mathcal N$ is flat over $\mathcal O_{F_p}$ with Cohen-Macaulay fibers. Moreover $\mathcal N$ is smooth in generic fiber, and 
$N$ is generically regular as $X_{h,h}$ is smooth for all $h$, thus $\mathcal N$ is R1 and S2, thus normal by Serre's criterion. 
The same is true for $Y$ using the local model diagram.

It only remain to prove the claim which is slightly more than what we did in the proof of Proposition \ref{propclosurestrata} as we don't assume lifts are in special fiber anymore.

\begin{proof}[Proof of claim] As before, $\Lambda$ is endowed with a $\mathcal O_{F_p} = \ZZ_p[X]/(X-\pi_1)(X-\pi_2)$-action, $X$ acting by $\pi$, and a pairing, and $\Lambda[\pi-\pi_i]=(\pi-\pi_{3-i})\Lambda$ is totally isotropic as, e.g.
\[ <(\pi-\pi_2)x,(\pi-\pi_2)y> = <\underbrace{(\pi-\pi_1)(\pi-\pi_2)}_{=0}x,y> = 0.\]
One can assume that the matrix of the modified pairing on $(\Lambda\otimes k) [\pi]$ is 
$$Q = \left ( \begin{array} {ccc}
0 & 0 & I_a \\
0 & I_{b-a} & 0 \\
I_a & 0 & 0
\end{array} \right )$$
where $\mathcal F_1$ is spanned by $\pi e_1, \dots, \pi e_a \in \Lambda \otimes k$, and $\mathcal F = (\Lambda\otimes k) [\pi]$. Let $R$ be a local $W_{\mathcal{O}} (k)$-algebra, with a surjective morphism $R \to k$, whose kernel is denoted $I$. We investigate the possible lifts of $\mathcal F$ to $\Lambda\otimes_{\ZZ_p} R$.
First one can assume that the matrix for the induced pairing between $\Lambda[\pi-\pi_1]$ and $\Lambda[\pi-\pi_2]$ is still given by the matrix $Q$, where $\Lambda[\pi-\pi_i]$ has the basis $(\pi - \pi_{3-i}) e_1,  \dots, (\pi - \pi_{3-i})e_{a+b}$. One can also assume that the matrix of the original pairing is 
$$ \left (  \begin{array}{cc}
0 & -Q \\
Q & 0
\end{array} \right )  $$
in the basis $(\pi - \pi_2) e_1, \dots, (\pi - \pi_2)e_{a+b}, e_1,  \dots, e_{a+b}$. \\
One should lift $\mathcal F_1$ inside $ \Lambda\otimes_{\ZZ_p} R[\pi - \pi_1]$. The latter is a $R$-module of rank $a+b$ with basis $(\pi - \pi_2) e_1, \dots, (\pi - \pi_2) e_{a+b}$. A lift of $\mathcal F_1$ is then spanned by the image of the matrix
$$\left ( \begin{array} {c}
I_a  \\
X  \\
Y
\end{array} \right ),$$
with $X,Y$ with coefficients in $I$.
Let $\widetilde{f_1}, \dots, \widetilde{f_a}$ be the vectors defined by the column of this matrix, and $\widetilde{e_1}, \dots, \widetilde{e_a}$ be the vectors defined by this matrix in the basis $e_1,  \dots, e_{a+b}$ (so that $\widetilde{f_a} = (\pi - \pi_2) \widetilde{e_a}$). One can then compute the orthogonal of $\mathcal F_1$ inside $ \Lambda\otimes_{\ZZ_p} R[\pi - \pi_2]$ (for the modified pairing). It is defined by the matrix
$$\left ( \begin{array} {cc}
I_a & 0 \\
0 & I_{b-a}  \\
- ^t Y & - ^t X
\end{array} \right ),$$
 in the basis $(\pi - \pi_1) e_1, \dots, (\pi - \pi_1) e_{a+b}$. In particular, a possible lift of $\mathcal F$ will automatically contain the image of the matrix $\left ( \begin{array} {c}
0 \\
I_{b-a}  \\
- ^t X
\end{array} \right )$ in the basis $(\pi - \pi_1) e_1, \dots, (\pi - \pi_1) e_{a+b}$. Let $\widetilde{\mathcal F_1}$ be the lift of $\mathcal F_1$, and $\widetilde{\mathcal F'}$ be the free $R$-module obtained as the sum of $\widetilde{\mathcal F_1}$ and the previous module. \\
Now, we are left to lift $\mathcal F$ inside $(\pi - \pi_2)^{-1} \widetilde{\mathcal F_1} / \widetilde{\mathcal F'}$. A basis for this module consists in $(\pi - \pi_1) e_{b+1}, \dots, (\pi - \pi_1) e_{a+b}, \widetilde{e_1}, \dots, \widetilde{e_a}$, and a lift of $\mathcal F$ will be given by the image of a matrix of the form $\left ( \begin{array} {c}
I_a \\
Z
\end{array} \right )$. All is left to do is to check the isotropy condition for the lift of $\mathcal F$. The last vectors are automatically orthogonal to the vectors obtained by the matrix $\left ( \begin{array} {c}
0 \\
I_{b-a}  \\
- ^t X
\end{array} \right )$ in $\Lambda\otimes_{\ZZ_p} R[\pi - \pi_2]$. The orthogonality with $ \widetilde{\mathcal F_1}$ gives the relation $(\pi_2 - \pi_1) I_a  + (Y + ^t Y + ^t X X) Z = 0$. Finally, the fact that these vectors should be pairwise orthogonal gives the relation $^t Z = Z$, hence the result.
\end{proof}
\end{proof} 

\section{Modular forms}
\label{section:modforms}
\subsection{Sheaves}

We define $\mathcal{E}_i = \mathcal{E} [T - \pi_i]$, for $i = 1, 2$.

\begin{prop}
The sheaf $\det \mathcal{E}$ is trivial, and one has $\det(\mathcal{E}_1) \simeq \det(\mathcal{E}_2)^{-1}$.
\end{prop}

\begin{proof}
The follows from the fact that $\mathcal E$ has an alternate pairing, that $\mathcal{E}_1$ is totally isotropic, and that the multiplication by $T - \pi_1$ induces an isomorphism $\mathcal{E} / \mathcal{E}_1 \simeq \mathcal{E}_2$. Indeed, $\det \mathcal E = \det \mathcal E_1 \otimes \det (\mathcal E/\mathcal E_1)$, and the map
\[ \mathcal E_1 \fleche \mathcal E \overset{\sim}{\underset{<,>}{\fleche}} \mathcal E^\vee \fleche \mathcal E_1^\vee,\]
is zero as $\mathcal E_1$ is totally isotropic. We thus deduce an isomorphism $\mathcal E_1 \overset{\sim}{\fleche} (\mathcal E/\mathcal E_1)^\vee$, and finally
\[ \det \mathcal E = \det \mathcal E_1 \otimes \det(\mathcal E_1)^{-1} = \mathcal O_S.\]
\end{proof}

\begin{prop}
One has isomorphisms 
$$\det(\mathcal{E}_1) \simeq \det(\omega_1 ) \otimes \det(\omega_2)^{-1} \simeq \det(\omega / \omega_2 ) \otimes \det(\omega / \omega_1)^{-1} $$
\end{prop}

\begin{proof}
Inside $\mathcal{E}$ the orthogonal of $\omega_1$ is $(T - \pi_1)^{-1} \omega_2$. This implies that $\det( (T - \pi_1)^{-1} \omega_2 / \mathcal{E}_1) \simeq \det (\mathcal{E}_1 / \omega_1)^{-1} \simeq  \det (\mathcal{E}_1)^{-1} \otimes \det( \omega_1)$. Moreover, the multiplication by $T - \pi_1$ induces an isomorphism between $(T - \pi_1)^{-1} \omega_2 / \mathcal{E}_1$ and $\omega_2$. \\
For the second part, one uses $\det \omega = \det \omega_1\otimes\det \omega/\omega_1 = \det \omega_2 \otimes \det \omega/\omega_2$.
\end{proof}

\begin{prop}
On the Rapoport locus, one has an isomorphism $\det(\omega / \omega_2) \simeq \det(\omega_1)$. \\
In the special fiber, one has an isomorphism $\mathcal{E}_1 \simeq \mathcal{E}_2$. The sheaf $\varepsilon := \det \mathcal{E}_1 $ satisfies $\varepsilon^2 \simeq \mathcal{O}_S$.
\end{prop}

\subsection{Definition and vanishing modulo $p$}

As we work in characteristic $p$, we will need to use an integral version of Schur functors.
See also \cite{Gold} section 3.8. For $\lambda$ a character of $\mathbb G_m^r$, and $M$ a rank $r$ free module over $R$, choose an isomorphism $M \simeq R^r$, and denote $\mathcal L(\lambda)$ the sheaf on $\GL_r/B$ ($B$ the upper triangular Borel of $\GL_r$) whose sections are given by
\[ \mathcal L(\lambda)(U) = \{ f : \pi^{-1}(U) \fleche \mathbb A^1 | f(gb) = \lambda^{-1}(b)(g) \forall b \in B,g \in \pi^{-1}(U)\}.\]
Denote by $\mathcal L_M(\lambda)$ the sheaf on the flag variety $\mathcal F\ell(M)$ for $M$, given by $\phi_*\mathcal L(\lambda)$  after choosing an isomorphism $\phi : R^r \simeq M$ (inducing $\GL_r \simeq \Isom_R(R^r,M)$ and $\phi : \GL_r/B \simeq \mathcal F\ell(M)$). This is independant of the choice of $\phi$.
For $\underline a = (a_1 \geq \dots \geq a_r) \in \ZZ^r$, with associated character of $T = \mathbb G_m^r \subset B$, denote $M^{(a_1,\dots,a_r)}$ the global sections of $\mathcal L_M(\underline a)$, i.e.
\[ M^{(a_1,\dots,a_r)} = H^0(\mathcal F\ell(M),\mathcal L_M(\underline a)).\]
As $H^1(\mathcal F\ell(M),\mathcal L_M(\underline a)) = 0$ (Kempf theorem, see \cite{Jantzen:RepAlgGrps} Proposition 4.5), the formation of $M^{(a_1,\dots,a_r)}$ commutes with base change $R \fleche R'$, 
and thus the construction glues to a functor from the category of rank $r$ vector bundles on a scheme $X$ to the category of vector bundles on $X$ (of any rank) associating to $\mathcal V$ or rank $r$ the vector bundle $\mathcal V^{(a_1,\dots,a_r)}$. We denote $\underline a^\vee = (-a_r,\dots,-a_1)$.

\begin{defin}
Let $k,l,r$ be three integers. A (scalar-valued) modular form of weight $(k,l,r)$ is a section of the sheaf
$$(\det \omega_1)^k \otimes (\det \omega / \omega_1)^l \otimes (\det \mathcal{E}_1)^r.$$
More generally, given $\underline{k}= (k_1,\dots,k_a) \in \ZZ^a, \underline{\ell} = (\ell_1,\dots,\ell_b)\in \ZZ^b$ with $k_1 \geq \dots \geq k_a, \ell_1 \geq \dots \geq \ell_b$, we can consider the sheaf 
\[ \omega^{(\underline{k},\underline{\ell},r)} := \omega_1^{\underline{k}} \otimes (\omega/\omega_1)^{\underline{\ell}} \otimes (\det \mathcal E_1)^r.\]  
A weight $(\underline{k},\underline{\ell},r)$ modular form is a section of this sheaf.
\end{defin}

\begin{rema}
In generic fiber we can remove the use of $r$, and we can replace $\omega/\omega_1$ by $\omega_2$. In special fiber though, $\omega_1 = \omega_2$ up to a square zero sheaf. In special fiber, we can moreover assume that $r = 0,1$.
\end{rema}
In special fiber we have the following vanishing result. 
\begin{prop}
If $-k_1,\dots,-k_a,-\ell_b,\dots,-\ell_1$ is not decreasing (i.e. if $k_1 > k_a$ or $k_a > \ell_b$), then
\[ H^0(\overline{X_{0,0}},\omega^{(\underline{k},\underline{\ell},r)}) = 0.\]
\end{prop}

\begin{proof}
Let $x \in X_{0,0}$. Then above $x$ we have $\omega_1 \subset \mathcal E[\pi] = \mathcal E_1 = \omega$. We look at $\Gr_{a,a+b}(\mathcal E[\pi])$ the Grassmanian of rank $a$ sub-bundles of $\mathcal E[\pi]$. Over it, we have a universal bundle $V_1 \subset \mathcal E[\pi]$, which induces an immersion $\Gr_{a,a+b}(\mathcal E_x[\pi]) \fleche \overline{X_{0,0}}$ mapping $\omega_1$ to $x$. The pullback of $\omega = \mathcal E[\pi]$ to $\Gr_{a,a+b}$ is constant, and the pullback of the universal $\omega_1$ on $\overline X$ is \[V_1 =: \mathcal O(\underbrace{-1,0,\dots,0}_{a \text{ times}},0,\dots,0)\](which corrresponds when $a = b = 1$ to $\mathcal O(-1))$ on $\mathbb P^1$ up to twist by center). Thus, the pullback of $\omega/\omega_1$ is \[ \mathcal E[\pi]/V_1 =: \mathcal O(0,\dots,0,\underbrace{0,\dots,0,-1}_{b \text{ times}}),\] (which corresponds to $\mathcal O(1)$ on $\mathbb P^1$ when $a=b=1$, up to twist by the center ). The restriction of a section of $\omega^{\underline k,\underline{\ell},r}$ to $\Gr_{a,a+b}$ is then $\mathcal L_P(-\underline k,\underline \ell^\vee)$. Remark that $\Gr_{a,a+b} \simeq P\backslash G$ for $G = \GL_{a+b}$ and $P$ the standard parabolic of size $a,b$. We thus have a map $B\backslash G \overset{\pi}{\fleche} P \backslash G$ (for the upper triangular Borel $B$), and
\[ \mathcal L_P(-\underline k,\underline \ell^\vee) = \pi_*\mathcal L(-k_1,\dots,-k_a,-\ell_b,\dots,-\ell_1),\]
with $\mathcal L(-k_1,\dots,-k_a,-\ell_b,\dots,-\ell_1)$ the line bundle on $B\backslash G$. But \[H^0(P\backslash G,\mathcal L_P(-\underline k,\underline \ell^\vee)) = 
H^0(B\backslash G,\mathcal L(-k_1,\dots,-k_a,-\ell_b,\dots,-\ell_1)) = 0\] under the assumption (see Proposition \ref{prop:H0vanish} and Lemma \ref{lemme:38}). This is true for all points of $X_{0,0}$ thus we have the vanishing result.
\end{proof}

The following is well known,
\begin{prop}[]
\label{prop:H0vanish}
Let $G$ be a split reductive group in characteristic $p$. Let $B \subset P \subset G$ be a Borel and a parabolic subgroup, and $T$ a torus of $B$. Denote $\pi : G \fleche G/B$ and $f : G/B \fleche G/P$. Let $\lambda \in X(T)$ be a weight. Let $\mathcal L(\lambda)$ be the line bundle on $G/B$ such that
\[ \mathcal L(\lambda)(U) = \{ f : \pi^{-1}(U) \fleche \mathbb A^1 | f(gb) = \lambda^{-1}(b)(g) \forall b \in B,g \in \pi^{-1}(U)\},\]
and $\mathcal L_P(\lambda) = f_*\mathcal L(\lambda)$. Then $\lambda$ is dominant if and only if
\[ H^0(G/P,\mathcal L_P(\lambda)) \neq 0.\]
\end{prop}

\begin{proof}
See \cite{Jantzen:RepAlgGrps} Section II.2 for the definitions. We have $H^0(G/P,\mathcal L_P(\lambda)) = H^{0}(G/B,\mathcal L(\lambda))$. But by \cite{Jantzen:RepAlgGrps} Proposition 2.6, $\lambda$ is dominant iff $H^0(G/B,\mathcal L(\lambda)) \neq 0$.
\end{proof}

\begin{lemm}
\label{lemme:38}
Let $G = \GL_{a+b}$, and $P$ a standard parabolic with Levi $\GL_a \times \GL_b$. Let $\mathcal W$ be the universal direct factor and $\mathcal V$ the universal quotient on $X = G/P$. Then $\mathcal W^{\underline k} \otimes \mathcal V^{\underline \ell}$ coincides with $\mathcal L_P(-\underline k,\underline \ell^\vee)$.
\end{lemm}

\begin{proof}
In particular we need to prove that $\mathcal V = \mathcal V^{(1,0,\dots,0)} = \mathcal L_P(0,\dots,0,-1)$ and $\mathcal W = \mathcal L_P(-1,0,\dots,0)$. But conversely, as $\mathcal L_P$ is compatible with tensor product (on sheaves) and sum (on characters, see \cite{Jantzen:RepAlgGrps} Chapter 4), and as Schur functors commutes with base change, it is enough to check this at the fiber over $1 \in G/P$ as a $P$ representation. But clearly $\mathcal L_P(0,\dots,0,-1)^{\underline \ell} = \mathcal L_P(0,\dots,0,-\ell_b,\dots,-\ell_1)$ as $P$-representation and similarly for $\mathcal L_P(-1,0,\dots,0)$. To prove that $\mathcal V$ coincides with $\mathcal L_P(0,\dots,0,1)$ and similarly for $\mathcal W$, recall that both are $G$-equivariant vector bundles, so we can check the isomorphism at the fiber above $1 \in G/P$. But it is clear that there $<e_1,\dots,e_a> = \mathcal W_1 = V_P(-1,0,\dots,0)$ as $P$-representation and similarly for $\mathcal V$.
\end{proof}

Now let $x \in X_{h,h}$ for some $h \leq a$. We have
\[ 0 \subset \pi\omega_x \subset \omega_1 \subset \omega_x[\pi] \subset \mathcal E[\pi].\]
In particular $\omega_1$ gives a point of $\Gr_{a-h,a+b-2h}(\omega_x[\pi]/\pi\omega_x),$ and there is a natural map \[\Gr_{a-h,a+b-2h}(\omega_x[\pi]/\pi\omega_x) \fleche \overline{X_{h,h}}.\] Moreover, the pullback of $\pi\omega,\omega[\pi]$ to the Grassmanian is constant (by construction) and thus we have an extension \[0 \fleche \pi\omega = \mathcal O^h \fleche \omega_1 \fleche \omega_1/\pi\omega = \mathcal O(0,\dots,0,-1) \fleche 0,\]
and
\[ 0 \fleche \omega[\pi]/\omega_1 = \mathcal O(-1,0,\dots,0) \fleche \omega/\omega_1 \fleche \omega/\omega[\pi] \simeq \pi\omega \fleche 0,\]
where the sheaves $\mathcal O(k_1,\dots,k_{a+b-2h})$ are on $\Gr_{a-h,a+b-2h}(\omega_w[\pi]/\pi\omega_x)$, with notations as before. Thus we can use the previous strategy to prove the following.

\begin{theor}
Assume $h < a$. If we cannot find $a-h$ indexes $i_t \in \{1,\dots,a\}$ and $b-h$ indexes $j_s \in \{1,\dots,b\}$ such that
\[ k_{i_1} = \dots = k_{i_{a-h}} \leq \ell_{j_1} \leq \dots \leq \ell_{j_{b-h}},\]
then 
\[ H^0(\overline{X_{h,h}},\omega^{(\underline k,\underline\ell,r)}) = 0.\]
\end{theor}

\begin{rema}
Note that this is the case in particular if $\underline k$ is regular enough, or if $h+1$ weights of $\underline k$ are greater than $\underline \ell$. The most restrictive case to apply the theorem is when $h = a-1$, in which case we can apply it under the assumption $k_a > \ell_{b-a+1}$.
\end{rema}

\begin{proof}
By what preceed, we can choose a point $x \in X_{h,h}$ and compute the global sections of $\omega^{(\underline k,\underline \ell,r)}$ on the associated Grassmanian $\Gr_x := \Gr_{a-h,a+b-2h}(\omega[\pi]/\omega)$ (seen as a closed subspace of $\overline{X_{h,h}}$). On this space, $\mathcal E$ is constant (the $p$-divisible group is fixed), thus we can forget about $r$. We denote the following subgroups of $\GL_{a+b-h}$ :
\[ M = 
\left(
\begin{array}{ccc}
  \GL_h &   & 0  \\
  & \GL_{a+b-2h}   & \\
0 & & \GL_h
\end{array}
\right) \supset P = 
\left(
\begin{array}{cccc}
  \GL_h &  & 0  \\
  & \GL_{a-h} & \star  & \\
 &0 & \GL_{b-h} & \\
0 & & & \GL_h
\end{array}
\right)
\]
and \[ P_{a-h,b-h} = \left(
\begin{array}{cc}
  \GL_{a-h} &  \star  \\
0  & \GL_{b-h}  
\end{array}
\right) \subset \GL_{a+b-2h}.\]
Clearly, we have an isomorphism $\Gr_{a-h,a+b-2h} := P_{a-h,b-h}\backslash \GL_{a+b-2h} \simeq P \backslash M =: \Gr$, and we will use the partial Borel-Weyl-Bott theorem on $P\backslash M = \Gr$. Denote $V$ the vector space of dimension $a+b$ on which $M$ acts, it corresponds to a vector bundle $\mathcal V$ on $\Gr$, which coincides with the pullback of $\omega$ to $\Gr$. As representation of $M$, $V = V_0 \oplus V_1 \oplus V_2$, a sum of irreducible, and we need to compute the weights of the representation $V^{\underline k,\underline \ell}$ (the Schur functor for $\GL_{a+b}$ associated to $(\underline k, \underline \ell)$) for the action of $P$. But as a representation of $GL_{a+b}$, $V^{\underline k,\underline \ell}$ has weights $w\cdot (k_1,\dots,k_a,\ell_1,\dots,\ell_b), w \in \mathfrak S_{a+b}$. Among those weights, the highest weights for the action of $P$ are those of the form $w_1w_2 \cdot (k_1,\dots,k_a,\ell_1,\dots,\ell_b)$ with $(w_1,w_2) \in \mathfrak S_a\times\mathfrak S_b$ and \[w_1(1)\geq \dots \geq w_1(h), \quad w_1(h+1) \geq \dots \geq w_1(a), \quad w_2(1)\geq \dots \geq w_2(b-h),\]\[ w_2(b-h+1) \geq \dots \geq w_2(b).\]
Denote ${^P}W$ this space.
Thus, $\mathcal V^{\underline k,\underline \ell}$ (and thus $\omega^{(\underline k,\underline \ell)}$) is an extension of $\mathcal L_P(w_1w_2 \cdot (-\underline k,\underline \ell^\vee))$ for $w \in {^P}W$. But under the hypothesis non of these bundles have sections (Proposition \ref{prop:H0vanish}), thus $H^0(\Gr,V^{\underline k,\underline \ell}) = H^0(\Gr_x,\omega^{\underline k,\underline \ell,r})= 0$. As this is true for any point $x \in X_{h,h}$, we deduce the result.
\end{proof}

\section{Further strata for the case $(1,n)$}
\label{sect:Hasse1n}
In this section, we consider the case where $(a,b) = (1,n)$, where $n \geq 1$ is an integer.

\subsection{Definition of the invariants}

We will define some invariants on the special fiber $X$. Let us recall that one has locally free sheaves $\omega_1, \omega_2$, of rank respectively $1$ and $n$.

\begin{defin}\label{defin:hasse}
We define $b \in H^0 (X, (\omega/\omega_2) \otimes \omega_1^{-1})$ thanks to the natural inclusion $\omega_1 \to \omega / \omega_2$. \\
We define $m \in H^0 (X, \omega_1 \otimes  (\omega/\omega_2)^{-1} $ thanks to the multiplication by $\pi : \omega / \omega_2 \to \omega_1$. \\
For $i \in \{ 1,2 \}$, we define $hasse_i \in H^0 ((\omega/\omega_i)^{(p)} \otimes \omega_1^{-1})$ thanks to the map $hasse: \mathcal{E} [\pi] \to (\omega/\omega_i)^{(p)}$, induced by the composition of the Verschiebung and the division by $\pi$.
\end{defin}

We refer to \cite{Bi_dual} Def. 3.8 for more details about the definition of the maps $hasse_i$ (note that the reference deal with the ordinary case i.e $a=b$).

\begin{prop}
We have the following properties.
\begin{itemize}
\item One has $bm = 0$ and $mb = 0$.
\item If $x$ is point of $X$ with $b(x) = 0$, then $hasse_1(x) = 0$ implies that $hasse_2 (x) = 0$.
\item If $x$ is point of $X$ with $b(x) \neq 0$, then one cannot have $hasse_1(x) = 0$ and $hasse_2 (x) = 0$.
\end{itemize}
\end{prop}

\begin{rema}
The stratification defined previously consists in three strata, according to whether the sections $b$ and $m$ are $0$ or not.
\end{rema}

\begin{proof}
Indeed clearly $mb = 0$ as $\pi \omega_1 = 0$. Moreover, as $\pi \omega \subset \omega_2$ (as $\omega_1 \subset (\pi\omega)^{\perp'}$ because $(\pi\omega)^{\perp'} = (\pi (\omega + \mathcal E[\pi]))^{\perp'} = (\omega + \mathcal E[\pi])^\perp$ and this last space contains $\omega_1$ as both $\omega$ and $\mathcal E[\pi]$ are totally isotropic), we have clearly that $bm = 0$.

For the second point, if $b = 0$ then $\omega_1 \subset \omega_2$ and thus if $\hasse_1 =0$, i.e. $\hasse(\omega_1) \subset \omega_1^{(p)}$ then $\hasse(\omega_1)\subset \omega_2^{(p)}$. For the last point, remark that if $x$ is a point, then $b \neq 0$ is equivalent to $\omega = \omega_1 \oplus \omega_2$ as $\omega_1$ is of rank 1. Thus the vanishing of both $\hasse_1$ and $\hasse_2$ is equivalent to the vanishing of $\omega_1 \overset{\hasse}{\fleche} \omega^{(p)}$. But because $\omega_1 \oplus \omega_2 = \omega$, which is thus of $\pi$-torsion, $\hasse$, which is surjective, induces an isomorphism $\mathcal E[\pi] \overset{\hasse}{\fleche} \omega^{(p)} = \mathcal E[\pi]^{(p)}$, and thus its restriction to $\omega_1$ can't be zero.
\end{proof}

Let us now define the different strata that we will consider.

\begin{itemize}
\item The ordinary locus is $X^{ord} = \{ x \in X, m(x) \neq 0, hasse_2(x) \neq 0 \}$.
\item $R_1 = \{ x \in X, m(x) \neq 0, hasse_1(x) \neq 0, hasse_2 (x) =0  \}$.
\item $R_2 = \{ x \in X, m(x) \neq 0, hasse_1 (x) =0  \}$.
\item $B_0 = \{ x \in X, b(x) \neq 0, hasse_1(x) \neq 0, hasse_2 (x) \neq 0 \}$. 
\item $B_1 = \{ x \in X, b(x) \neq 0, hasse_2(x)  =0  \}$.
\item $B_2 = \{ x \in X, b(x) \neq 0, hasse_1 (x) =0  \}$.
\item $P_0 = \{ x \in X, m(x) = b(x) = 0, hasse_2 (x) \neq 0  \}$.
\item $P_1 = \{ x \in X, m(x) = b(x) = 0, hasse_1(x) \neq 0, hasse_2 (x) =0  \}$.
\item $P_2 = \{ x \in X, m(x) = b(x) = 0, hasse_1 (x) =0  \}$.
\end{itemize}

\begin{prop}
Let $x$ be a point in $X^{ord}$. Then $x$ is $\mu$-ordinary in the sense of \cite{BH}. In particular, one has $A[\pi] \simeq \mu_p \times \mathbb{Z}/p\mathbb{Z} \times LT^{n-1}$. 
\end{prop}

\begin{rema}
Here $LT$ is defined in \cite{BH} before Définition 1.1.3, this is $X_\beta$ with $\beta = (1)$ ($e = 2$ and $\mathcal T$ is a singleton).
\end{rema}

\subsection{The conjugate filtration}

The Verschiebung induces a map $V : \mathcal{E} \to \omega^{(p)}$, which is compatible with the action of $\pi$.

\begin{defin}
We define the sheave $\mathcal{F}_i$, $i=1,2$ by the formula
$$\mathcal{F}_i := \pi \cdot V^{-1} \omega_i^{(p)}$$ 
\end{defin}

\begin{prop}
The sheaves $\mathcal{F}_i$, $i=1,2$ are locally free of rank $1$ and $n$, and are included in $\mathcal{E} [\pi]$. Moreover, $\mathcal{F}_2$ is the orthogonal of $\mathcal{F}_1$ for the modified pairing.
\end{prop}

\begin{proof}
The sheaf $V^{-1} \omega_1^{(p)}$ is locally free of rank $n+2 = a+b+1$, and contains $\mathcal{E} [\pi]$. This implies that $\mathcal{F}_1$ is locally free of rank $1$. One gets in a similar way the result for $\mathcal{F}_2$. \\
To prove the last part, one only needs to check that $\mathcal{F}_1$ and $\mathcal{F}_2$ are orthogonal. Let $x \in \mathcal{F}_1$ and $y \in \mathcal{F}_2$. By definition, there exist $x',y'$ such that $x = \pi x'$ and $y = \pi y'$, and $Vx' \in \omega_1^{(p)}$,  $Vy' \in \omega_2^{(p)}$. Since $\omega_1$, and $\omega_2$ are orthogonal for the modified pairing, one gets the relation $\{ Vx', Vy' \} = 0$. The element $Vx'$ is in $\mathcal{E} [\pi] ^{(p)}$; there exists then $z \in \mathcal{E}^{(p)}$ such that $Vx' = \pi z$. Now one has
$$0 = \{ Vx', Vy' \} = \{\pi z, Vy' \} = <z, Vy'> = <Fz, y'>$$
But there exists a unit $u$ such that $u Fz  = \pi x' = x$, this equality being in $\mathcal{F} / \pi \mathcal{F}$, where $\mathcal{F} = Ker V$. There exists then $a \in \mathcal{F}$ such that $Fz = u^{-1} x + \pi a$.  Thus 
$0 = <u^{-1} x + \pi a,y'> = \{u^{-1} x,y \} - < a,y> = \{u^{-1} x,y \}$. Indeed, since $a$ and $y$ belong to $\mathcal{F}$, which is totally isotropic, one must have $<a,y> = 0$. One then observes that the quantity $\{u^{-1} x , y \}$ only depends on the class of $u$ in $O_F / \pi$, and one concludes that $\{x,y\} = 0$.
\end{proof}

\begin{prop}
Let $x$ be a point of $X$. Then the condition $hasse_2 (x) = 0$ is equivalent to $\omega_1 \subseteq \mathcal{F}_2$. The condition $hasse_1 (x) =0$ is equivalent to $\omega_1 = \mathcal{F}_1$.
\end{prop}

\subsection{Stratification when $n >1$}

First, we remark that $R_1$ and $P_1$ are empty if $n \leq 2$. 

\begin{prop}
Assume that $n \leq 2$. Then $R_1$ and $P_1$ are empty.
\end{prop}

\begin{proof}
Assume that $x$ is a point in $R_1$ or $P_1$. This implies that $\omega_1 \subseteq \mathcal{F}_2$. If $n=1$, since $b(x) =0$, one must have $\omega_1 = \omega_2$, hence $\mathcal{F}_1 = \mathcal{F}_2$ and then $hasse_1 (x) =0$. This is a contradiction. \\
Assume now that $n=2$. Taking the orthogonal of the inclusion $\omega_1 \subseteq \mathcal{F}_2$ in $\mathcal{E} [\pi]$, one has $\mathcal{F}_1 \subset \omega_2$. As $b = 0$ we have $\omega_1 \subset \omega_2$, thus $\omega_1^{(p)} \subset \omega_2^{(p)}$ and thus $\mathcal F_1 \subset \mathcal F_2$. In particular $\omega_1$ and $\mathcal{F}_1$ are distincts isotropic lines, and $\omega_2$ is the orthogonal of $\omega_1$, thus one can see that the modified pairing induced on $\omega_2$ is zero, which is not possible.
\end{proof}

Let us now state the principal result on the stratification of the variety. We will need the following remark.

\begin{rema}
\label{remaHdgfrob}
Let $S_0 = \Spec(R)$ be a characteristic $p$ scheme, and $T = \Spec(S)$, with $R = S/I$ for some ideal $I$ a thickening of $S_0$, and assume $T$ is of characteristic $p$ again. Let 
$G$ be a p-divisible group over $S_0$ and assume that $I^2 = 0$ in $T$ and denote $\mathcal E$ its crystal on the crystalline site $S_0/\Spec(\ZZ_p)$. Then by Grothendieck-Messing,
 lifting $G$ to $T$ is the same as lifting is Hodge filtration $\omega_G$ to $\mathcal E_T$. Assume $\widetilde \omega_G \subset \mathcal E_T$ is such a lift, then as $I^2 = 0$ we 
 claim that $\widetilde \omega_G^{(p)}$ doesn't depend on the lift. Indeed, let $w_1,w_2 \in \mathcal E_T$ which both lift $w \in \mathcal E_{S_0}$ and let $\underline e$ be a basis of 
 $\mathcal E_{T}$ as $S$-module. Then $w_2 = w_1 + M\cdot \underline e$ for some $M \in M_{2h}(I)$. Then 
 $w_2 \otimes 1 = w_1\otimes 1 + (M \underline e)\otimes 1 = w_1\otimes 1 + \underline e \otimes M^\sigma$. But if $i \in I$, and $\sigma = \sigma_T$ is the Frobenius of $T$ 
 (which lifts the one of $S_0$) then $\sigma(i) = i^p \equiv 0$ in $S$, thus $w_2 = w_1$. In particular, in the previous situation as both $F,V$ are maps on the crystal $\mathcal E$, 
 we see that the lifts of $\mathcal F_1,\mathcal F_2$ doesn't depend on the lift of $\omega$.
\end{rema}

\begin{theor}
Assume that $n \geq 2$. The strata $X^{ord}$ and $B_0$ are open. The strata $P_2$, $B_2$ are closed. Moreover
$$\overline{X^{ord}} = X^{ord} \cup_{i=1}^2 R_i \cup_{i=0}^2 P_i \qquad  \overline{R_2} = R_2 \cup P_2$$

$$\overline{B_0} = \cup_{i=0}^2 B_i \cup_{i = 0}^2 P_i \qquad \overline{B_1} = B_1 \cup P_1 \cup P_2 \qquad \overline{P_0} = \cup_{i=0}^2 P_i $$
If $n \geq 3$, one has moreover
$$\overline{R_1} = \cup_{i=1}^2 R_i \cup_{i=1}^2 P_i \qquad \overline{P_1} = P_1 \cup P_2$$
\end{theor}

\begin{proof}
The fact that $X^{ord}$ and $B_0$ are open is clear, as is the closeness of $P_2$. Let us prove the closure relations by looking at where we can specialize (for $B_2$) or deform points of $X$.\begin{itemize}
\item If $x$ is a point of $B_2$, it can only specialize to a point in $B_2$ or $P_2$. We need to show that the latter cannot happen. Assume that a point $x$ in $P_2(k)$ can be deformed to $k[[X]]$\footnote{In particular this means that we have a deformation of the Hodge filtration, and conversly a deformation of the Hodge filtration to $k[[X]]$ induces step by step by Grothendieck-Messing a deformation of the $p$-divisible group.}, such that the generization lies in $B_2$. Since $hasse_1 =0$, $\omega_1 = \mathcal{F}_1$ over $k[[X]]$. If $e_1$ is a basis of $\omega_1$, let $u = \{e_1, e_1\}$. The composition of $V$ with the division by $\pi$ defines a map $V_\pi : \mathcal{E} [\pi] \to \mathcal{E} [\pi]^{(p)}$; similarly, one has a map $F_\pi : \mathcal{E} [\pi]^{(p)} \to \mathcal{E} [\pi]$ given by the composition of the division by $\pi$ and the Frobenius. These maps are well defined because the image of $V$ is $\mathcal{E} [\pi]$. There exists a unit $u \in O_F^\times$ such that $F_\pi \circ V_\pi = u \id$. Since one has $\{F_\pi x,y\} = \{x, V_\pi y \}$, and $V_\pi e_1 = \lambda e_1$ for some unit $\lambda$, one finds the equation $u = \lambda_0 u^p$
, with $\lambda_0 \in k^\times$. One gets a contradiction, since $u$ must be non zero and divisible by $X$. \\
\item Let $x \in R_2(k)$. This implies that $\omega_1 = \mathcal{F}_1$, and $\omega_2 = \mathcal{F}_2$. One can find a basis $e_1, \dots, e_{n+1}$ of $\mathcal{E} [\pi]$ such that $\omega_1$ is spanned by $e_1$, and $\omega_2$ by $e_1, \dots, e_n$, and the modified pairing is given by the matrix
\begin{equation}\label{formpairing}\left ( \begin{array} {ccc}
0 & 0 & 1 \\
0 & J_{n-1} & 0 \\
1 & 0 & 0
\end{array} \right ), \quad \text{with } J_{n-1} = \left ( \begin{array} {ccc}
0 & 0 & 1 \\
0 & \reflectbox{$\ddots$} & 0 \\
1 & 0 & 0
\end{array} \right ).\end{equation}
One then look for a lift to $k[[T]]$, first of the Hodge filtration together with the extra data. The line $\omega_1$ can be lifted to a line spanned by a vector $\left ( \begin{array}{c}
1 \\
X \\
y 
\end{array} \right)$. The vector needs to be isotropic, hence the condition
$$2y + ^t XJ_{n-1}X =0$$
Then we will look at the corresponding deformation step by step, i.e. successively from $k[T]/(T^n) \fleche k[T]/(T^{n-1})$ which is given by a square zero ideal. At each step, 
we have a $p$-divisible group $G_n$ over $k[T]/(T^{n})$ and by remark \ref{remaHdgfrob}, the deformation to $G_{n+1}$ has a canonical lift of 
$\mathcal F_1$ and $\mathcal F_2$ which we can assume, if $\hasse_1(G_n) = \hasse_2(G_n) = 0$ given by $e_1$ and $e_1,\dots,e_n$. The condition for the generization 
to be in $R_2$ is that at each step $X=0$, $y=0$. The condition for it to be in $R_1$ is $y=0$. Since $n \geq 2$, the point can always be lifted to a point in $X^{ord}$. If 
$n \geq 3$, it can be lifted to a point in $R_1$, but as in this case the two conditions ($\omega_1$ totally isotropic and $\omega_1 \subset \mathcal F_2$) make a non smooth 
condition, let us give a more precise argument : set $\widetilde{\mathcal E} = \mathcal E \otimes_k k[[t]]$ and choose a lift of the basis such that the pairing is of the previous form, and reducing on $k[t]/(t^2)$ we have $\mathcal F_1$ given by $e_1$ and $\mathcal F_2$ by $e_1,\dots,e_n$, as before. Then set $\widetilde\omega_1$ spanned by $\left ( \begin{array}{c}
1 \\
t \\
 0_{n-1} 
\end{array} \right)$. Then clearly $\widetilde\omega_1$ is totally isotropic, and reducing modulo $t^2$ we see that $\widetilde\omega_1 \neq \mathcal F_1 \mod t^2$, thus our deformed point is not in $R_2$ anymore, but we can't assure that the $k[[t]]$-point is in $R_1$ at the moment. So assume we have lifted $\omega_1$ to $k[t]/(t^n)$ to a point in $R_1$, and we moreover assume that there is a basis of $\mathcal E \otimes_k k[t]/(t^n)$ such that $\omega_1$ is spanned by $\left ( \begin{array}{c}
1 \\
t \\
 0_{n-1} 
\end{array} \right)$. We then choose a lift of this basis to $\mathcal E \otimes_k k[t]/(t^{n+1})$ such that the pairing as the same form (\ref{formpairing}). Then we simply set again $\widetilde\omega_1$ to be spanned by $\left ( \begin{array}{c}
1 \\
t \\
 0_{n-1} 
\end{array} \right)$. Inducting the argument gives the resulting point in $R_1$. 
\item The fact that any point $x \in P_0(k)$ can be deformed to a point in $X^{ord}$ or $B_0$ follows from the previous sections. 
\item A point $x \in P_2(k)$ can be deformed to $P_0$ (and hence $B_0$, $X^{ord}$), and $P_1$ if $n \geq 3$, with exactly the same arguments as before, as we never used that $m \neq 0$. If we want to deform $x$ to $R_2$, we can lift $\omega_1$ "trivially" so that $\omega_1 \subset \omega_2 := \omega_1^{\perp'}$, and then deform $\omega/\omega_2$ so that $\pi \omega \subset \omega_1$ (by choosing elements as in Proposition \ref{propclosurestrata}). We can then deform to $R_1$ if $n \geq 3$. The point $x$ can also be deformed to $B_1$, by lifting $\omega_1$ inside $\mathcal{F}_2$, non isotropically : concretely choose the isotrivial lift to $k[[T]]$ of $\widetilde{\omega}$ of $\omega$ (here it means it is still of $p$-torsion, i.e. $\widetilde \omega = \widetilde{\mathcal E}[\pi]$), and then inductively for each $n$, there is a canonical lift of $\mathcal F_1,\mathcal F_2$ from $k[T]/(T^{n-1})$ to $k[T]/(T^n)$ by remark \ref{remaHdgfrob} if $\omega_1$ and thus $\omega_2 = \omega_1^{\perp'}$ have been deformed to $k[T]/(T^{n-1})$. We thus have deformations of $\mathcal F_1, \mathcal F_2$ to $k[T]/(T^n)$ (orthogonal to each other for the modified pairing) and we choose a deformation of $\omega_1$ still assuming $\widetilde \omega_1 \subset \mathcal F_2$. This is possible as the Grassmanian $\Gr_a(\mathcal F_2)$ is smooth. If $n$ is big enough, as the condition of being totally isotropic is a closed condition which defines a proper closed subspace of $\Gr_a(\mathcal F_2)$, there exists a deformation of $\omega_1$ which is not isotropic anymore. After this choice, any lift of $\omega_1$ will do. The corresponding deformed $p$-divisible group is in $B_1$. Note that we have already proven that we can't deform from $P_2$ to $B_2$. 
\item  A similar but easier argument shows that we can deform from $B_2$ to $B_1$, and
it is easy to see that any element $x \in B_1 (k)$ 
can be deformed to $B_0$. 
\item To finish the proof, let us remark that a point in $R_1$ can be deformed to $X^{ord}$ if $n \geq 3$, by lifting $\omega_1$ isotropically outside $\mathcal{F}_2$. 
Indeed, we have $\pi \omega = \omega_1 \subset \omega_2$ and by hypothesis $\mathcal F_1 \neq \omega_1 \subset \mathcal F_2$. In particular $\omega_2 \neq \mathcal F_2$. The divided pairing, on a basis $e_1,\dots,e_h$ such that $e_1$ generates $\omega_1$ and $e_1,\dots,e_{h-1}$ generates $\omega_2$ can be given by 
\[\left ( \begin{array} {ccc}
0 & 0 & 1 \\
0 & I_{n-1} & 0 \\
1 & 0 & 0
\end{array} \right )\]
We thus look for a lift of $\omega_1$ given by a vector $\left ( \begin{array}{c}
1 \\
X \\
y 
\end{array} \right)$ with $y,X$ with coefficients in $tk[[t]]$. This lift is totally isotropic if $2y + {^t}XX = 0$. Let us prove that we can choose it away from $\mathcal F_2$. 
As mod $t$, $\omega_1 \subset \mathcal F_2 \neq \omega_2$, we have $e_1 \in \mathcal F_2$, there exist $e_i \notin \mathcal F_2$ and as $n \geq 2$, there is a non zero vector of the form
\[ v = \left ( \begin{array}{c}
0 \\
B \\
0 
\end{array} \right) \in \mathcal F_2.\] 
Thus if we set $\widetilde\omega_1$ generated by
\[ v = \left ( \begin{array}{c}
1 \\
tB + ta \delta_i \\
0 
\end{array} \right) \notin \mathcal F_2.\] 
for a non zero $a$, the condition of being totally isotropic is given by $t^2(\sum_{j} b_j^2 + 2b_ia + a^2)=0$. In particular if $\sum_j b_j^2$ is non zero or if $b_i \neq 0$ we can find such a non zero $a$. So assume that $\sum_j b_j^2 = 0$ but $b_i = 0$. As $v$ is non zero, there is $j$ such that $b_j \neq 0$. If $e_j \notin \mathcal F_2$ then the previous argument applies. Otherwise $e_j \in \mathcal F_2$, and thus $w = v + c e_j \in \mathcal F_2$. But if we calculate its norm for the divided pairing, we have $\sum_i b_i^2 + 2cb_j + c^2 = 2cb_j + c^2$. But we can find $c$ such that this is non zero, and then reapply the previous argument with $w$ instead of $v$. 
 
\item Finally, if $n \geq 3$, one checks that a point in $P_1$ can be deformed to $P_0$ (and then $X^{ord}$, $B_0$) by the exact same calculation. We can also deform from $P_1$ to $B_1$ :  mod $t$ we have $\omega_1 \subset \omega_2 \neq \mathcal F_2$, thus up to choosing a basis as before we can set $\widetilde w_1$ mod $t^2$ generated by
\[  \left ( \begin{array}{c}
1 \\
tB \\
t 
\end{array} \right) \in \widetilde{\mathcal F_2},\] where
\[ v = \left ( \begin{array}{c}
0 \\
B \\
1 
\end{array} \right) \in \mathcal F_2.\]  Then clearly $\widetilde \omega_1$ is not isotropic. Then assume that we have lifted $\omega_1$ to $k[t]/(t^n)$, this gives a lift of $\mathcal F_2$ to 
$k[t]/(t^{n+1})$ and we choose any lift of $\omega_1$ inside this. By induction, and Grothendieck-Messing, we get a point in $B_1$. 
We can also deform from $P_1$ to $R_1$ : assume that we have lifted $\omega_1$ to $k[t]/(t^n)$, inside $\mathcal F_2$, which has a canonical lift mod $t^{n+1}$. Then we want to deform $\omega_1$ isotropically while staying in $\mathcal F_2$. But as $\omega_1^\perp \cap \mathcal F_2$ is non trivial in special fiber, we can indeed find a lift of $\omega_1 \subset \mathcal F_2$ at each step which remains isotropic. 

\end{itemize}
\end{proof}

\subsection{Stratification when $n=1$}

We now suppose that $n=1$. In this case $P_1$ and $R_1$ are empty. The situation is the following.

\begin{theor}
The strata $X^{ord}$, $R_2$ and $B_0$ are open. The strata $P_0$, $P_2$ and $B_i$ ($i= 1,2$) are closed. Moreover
$$\overline{X^{ord}} = X^{ord} \cup P_0 \qquad \overline{R_2} = R_2 \cup P_2 \qquad \overline{B_0} = \cup_{i=0}^2 B_i \cup P_0 \cup P_2$$
\end{theor}

\begin{proof}
It is clear that $X^{ord}$ and $B_0$ are open. As previously, any point of $P_0$ can be deformed to $X^{ord}$ or $B_0$. It is easy to see that any point of $B_i$ ($i=1,2$) can be deformed to $B_0$. \\
Let us prove that $R_2$ is open. Let $x \in R_2 (k)$, and let us investigate the possible lifts of $x$ to a ring $R$. Over this ring, the space $\mathcal{F}_1$ lifts canonically. By assumption, the space $\omega_1$ is equal to its $\mathcal F_1$ over $k$. Since any lift of $\omega_1$ must be isotropic, and $\mathcal E[\pi]$ is a 2-dimensional space with a perfect pairing, we see that the space of totally isotropic lines in it is zero dimensional and reduced, thus one must have an equality $\widetilde\omega_1 = \mathcal F_1$. It is then not possible to lift $x$ to a point in $X^{ord}$. \\
The same arguments show that a point in $P_2$ cannot be deformed into $P_0$ or $X^{ord}$. Similarly, if $x \in P_2$ is deformed over $k[[t]]$ in $B_1$ or $B_2$, for each $n$, modulo $t^n$ this implies that we have canonical lifts $\mathcal F_1,\mathcal F_2$ modulo $T^{n+1}$. If $\omega_{1} = \omega_2$ mod $t^n$, then $\mathcal F_1 = \mathcal F_2$, and if we deform in $B_1$ or $B_2$ (or any point such that $\hasse_2 = 0$) we must have $\widetilde \omega_1 \subset \mathcal F_2$, but they have the same rank thus an equality, and thus $\widetilde\omega_1 = \mathcal F_1 = \mathcal F_2 = \widetilde\omega_2$. Thus actually the deformation remain in $P_2$.
This proves that points of $P_2$ can only be possibly deformed to a point in $R_2$ or $B_0$. Conversely we can indeed deform to $R_2$ by only deforming $\omega/\omega_2$ to make it non-$\pi$-torsion as in proof of proposition \ref{propclosurestrata}. To deform a point of $P_2$ to $B_0$, it is enough to deform $\omega_1 \subset \omega = \mathcal E[\pi]$ by a non-totally isotropic line. This is possible as this space is smooth (it is a projective space of dimension $> 0)$.
\end{proof}

\section{Case of a general CM field $F$}

Let $(B,\star,V,<,>,h)$ be P.E.L. datum (see \cite{Lan}), so that $B/\QQ$ be a finite dimensional central semi-simple $\QQ$-algebra, with involution $\star$, center $F$. 

\begin{example}
Let $F_0$ be a totally real field, and $F/F_0$ a CM field. Take $B = F$, $\star = c$ the complex conjugacy, $V = F^n$ and polarisation by $(x,y) = xJc(y)$ for an hermitian matrix $J$. Let $p$ be a prime. Then $B_{\QQ_p} = \prod_{\pi_0 | p \in F} F \otimes_{F_0} F_{0,\pi}$. Everything splits over primes above $p$ in $F_0$, thus for simplicity, we can assume that there is only one prime $\pi_0$ of $F_0$ above $p$. Let $e,f$ be the ramification index, and the residual degree of $\pi_0$ over $p$. The case of unramified primes in $F/F_0$ is treated in \cite{BH2}, thus we can assume that $\pi_0$ ramifies in $F_p := F\otimes\QQ_p$ and choose it so that $\pi_0 = \pi^2$ for some uniformizer $\pi$ of $F_p$.
\end{example}

Now fix an integral P.E.L. datum $(\mathcal O_B,\star,\Lambda,<,>)$, so that in particular $\mathcal O_B$ is a $\ZZ_{(p)}$-order in $B$, $\star$-stable and maximal over $\ZZ_p$, and $(\Lambda,<,>) \otimes_\ZZ \QQ = (V,<,>)$. 
\begin{hypothese} We assume the following
\begin{enumerate}
\item $B_{\QQ_p}$ is a product of matrix algebra over finite extension of $\QQ_p$.
\item $p$ is a \emph{good} prime, i.e. $p \nmid [\Lambda^\sharp,\Lambda]$.
\end{enumerate}
\end{hypothese}

To simplify we assume that $\star$ is of the second kind on each simple factors of $(B,\star)$ (in particular we exclude factors of type $D$ see \cite{BH2}, Hypothesis 2.2) : factors 
of type (C) can be dealt with as in \cite{BH2}. In most of what follows, we can treat simple factors separately, so that we will be able to assume $(B_{\QQ_p},\star)$ is a matrix 
algebra over its center or a product of two isomorphic matrix algebras over a field exchanged by $\star$. This second case is treated in \cite{BH2}. So to fix notations we will often 
assume that $B_{\QQ_p} = M_n(F_\pi)$, for some finite extension $F_\pi$ of $\QQ_p$, and we denote $F_{\pi_0} = F_\pi^{\star = 1}$ : the extension $F_\pi/F_{\pi_0}$ is of 
degree 2. 
Moreover, we assume that the local field extension $F_\pi/F_{\pi_0}$ is ramified (otherwise this is treated in \cite{BH2} again). Fix an uniformizer $\pi_0$ of $F_{\pi_0}$ and 
$\pi$ of $F_\pi$ so that $\pi^2 = \pi_0$.
 Let $F_{\pi_0}^{ur}$ be the maximal unramified extension contained in $F_{\pi_0}$, and $\mathcal{T}$ the set of embeddings of $F_{\pi_0}^{ur}$ into $\overline{\QQ_p}$. For each $\tau \in \mathcal{T}$, let $\Sigma_\tau$ be the set of embeddings of $F_{\pi_0}$ extending $\tau$. We write $\Sigma_\tau = \{ \sigma_{\tau,1}, \dots, \sigma_{\tau,e}\}$. For each $\tau \in \mathcal{T}$ and $1 \leq i \leq e$, let $\sigma_{\tau,i}^{+}$ and $\sigma_{\tau,i}^{-}$ be the embeddings of $F_\pi$ extending $\sigma_{\tau,i}$ : these are notations which will remain in force everytime we (implicitely) choose a simple factor of $B_{\QQ_p}$.
 
 As in \cite{BH2}, Section 2.2, we can associate to the Shimura data a combinatorial data $(d_{j,\tau'})_{\tau'}$, ($j$ corresponding to a choice of a simple factor) which, when we restrict to a simple factor of the previous type, is just a collection $(d_{\sigma})_{\sigma : F_{\pi} \hookrightarrow \overline{\QQ_p}}$ satisfying $d_{\sigma \circ c} = h-d_{\sigma}$ for a fixed value of $h \geq 1$ (which might depend on the simple factor). For simplicity we denote for $\tau \in \mathcal T, i =1,\dots,e$, $a_{\tau,i} = d_{\sigma_{\tau,i}^+}, b_{\tau,i} := d_{\sigma_{\tau,i}^-}$.

\begin{defin}
Let $Y$ be the moduli space over $O_{F_\pi}$ whose $R$-points are equivalence classes of tuples $(A, \lambda, \iota, \eta, \omega_1)$ up to $\ZZ_{(p)}^\times$-isogenies, where
\begin{itemize}
\item $A$ is an abelian scheme over $R$
\item $\lambda$ is a $\ZZ_{(p)}^\times$-polarization
\item $\iota : O_B \to End(A)\otimes_\ZZ \ZZ_{(p)}$, making the Rosati involution and $\star$ compatible
\item $\eta$ is a rational $\Lambda$-level structure outside $p$
\item For every simple factor $j = M_n(F_\pi)$ of $B_{\QQ_p}$, there is an associated direct factor $\omega'_j$ of $\omega_A$. By Morita equivalence, we have a $\mathcal O_{F_{\pi}}$-module $\omega_j = \bigoplus_{\tau} \omega_{\tau,j}$. We then ask for
\[0 = \omega_\tau^{[0]} \subseteq \omega_\tau^{[1]} \subseteq \dots  \subseteq \omega_\tau^{[e]} = \omega_{\tau,j},\] is a PR filtration, meaning that each $\omega_\tau^{[i]}$ is locally a direct factor, stable by $O_{F_\pi}$.
\item the quotient $\omega_\tau^{[i]} / \omega_\tau^{[i-1]}$ is locally free of rank $h$.
\item $O_{F_{\pi_0}}$ acts by $\sigma_{\tau,i}$ on $\omega_\tau^{[i]} / \omega_\tau^{[i-1]}$.
\item the filtration is compatible with the polarization
\item For each $i$, $\omega_{\tau}^{[i-1]} \subseteq \omega_{\tau,1}^{[i]} \subseteq \omega_\tau^{[i]}$, where $\omega_{\tau,1}^{[i]}$ is locally a direct factor stable by $O_{F_\pi}$.
\item $\omega_{\tau,1}^{[i]} / \omega_{\tau}^{[i]}$ is locally free of rank $a_{\tau,i}$, and $O_{F_\pi}$ acts by $\sigma_{\tau,i}^{+}$ on it, and by $\sigma_{\tau,i}^-$ on the quotient $\omega_\tau^{[i+1]}/\omega_{\tau,1}^{[i]}$ (which is automatically locally free of rank $b_{\tau,i}$).
\end{itemize}
\end{defin}

Let us be more precise about the compatibility with the polarization. One has a pairing on $\mathcal{E}$, and $\omega_{\tau,j}$ is totally isotropic for this pairing. The compatibility for the filtration is that
\[(\omega_\tau^{[i]})^\bot = Q_\tau^{i}(\pi_0)^{-1} \omega_\tau^{[i]}, \quad Q_\tau^i(T) = \prod_{t = i+1}^e (T-\sigma_{\tau,t}(\pi_0)), \] and $Q_\tau = \prod_{i=1}^e (T-\sigma_{\tau,i}(\pi_0))$ is a minimal polynomial for $\pi_0$ in $\tau(F_{\pi_0}^{ur})$.
Let us define $\mathcal{E}_\tau^{[i]} := (\pi_0 - \sigma_{\tau,i}(\pi_0))^{-1} \omega_{\tau}^{[i-1]} / \omega_{\tau}^{[i-1]}$. It is a locally free sheaf with an action of $O_F$, and an alternating perfect pairing. One has the subsheaves $\mathcal{F}_{\tau}^{[i]} := \omega_{\tau}^{[i]} / \omega_{\tau}^{[i-1]}$, which is totally isotropic for the previous pairing, and $\mathcal{F}_{\tau,1}^{[i]} := \omega_{\tau,1}^{[i]} / \omega_{\tau}^{[i-1]}$. We define $\mathcal{F}_{\tau,2}^{[i]} := (\pi - \sigma_{i,\tau}^-(\pi)) (\mathcal{F}_{\tau,1}^{[i]})^\bot$. 

\begin{defin} Let $k$ be an algebraically closed field of characteristics $p$.
Let $x \in Y(k)$, and let $\tau,i$. We define the integers $h_{\tau}^{[i]}$ and $l_{\tau}^{[i]}$ as the dimensions of $\pi \mathcal{F}_{\tau}^{[i]}$ and $\mathcal{F}_{\tau,1}^{[i]} \cap \mathcal{F}_{\tau,2}^{[i]}$ respectively.
\end{defin}

Let $C := \{ (h_{\tau}^{[i]}, l_{\tau}^{[i]})_{\tau \in \mathcal T, 1 \leq i \leq e}, 0 \leq h_{\tau}^{[i]} \leq l_{\tau}^{[i]} \leq \min (a_{\tau,i}, b_{\tau,i}) \}$. We define a stratification on $X = Y \times \Spec(k_F)$ by
\[X = \coprod_{c \in C} X_c,\]
where $X_c = \{ x \in X(k) | (h_\tau^{[i]}(x),l_{\tau}^{[i]}(x)) = c \}$.
Let $c=(h_{\tau}^{[i]}, l_{\tau}^{[i]})$ and $c' = ({h_{\tau}^{[i]}}', {l_{\tau}^{[i]}}')$ be elements of $C$. We say that $c \leq c'$ if for all $\tau, i$,
$$h_{\tau}^{[i]'} \leq {h_{\tau}^{[i]}} \leq {l_{\tau}^{[i]}} \leq l_{\tau}^{[i]'}$$

\begin{theor}
\label{thm:closurerelations}
One has 
$$\overline{X_c} = \coprod_{c' \leq c} X_{c'}$$
\end{theor}

\begin{proof}
When deforming a point of $Y$, one has to deform the Hodge filtration. We do it one $\tau$ at a time as follows. \\
By the results of the previous section, we can deform both $\omega_{\tau,1}^{[1]} \subseteq \omega_\tau^{[1]}$ inside $\widetilde{\mathcal E}_{\tau}^{[1]} := \mathcal{E}_\tau^{[1]} \otimes_k k[[t]]$, with the deformation $\widetilde{\omega}_\tau^{[1]}$ of $\omega_\tau^{[1]}$ isotropic (for the divided pairing) and with $(h_{\tau}^{[1]},l_\tau^{[1]}) = (h_{\tau}^{[1]'},l_\tau^{[1]'})$. This is the result of Proposition \ref{propclosurestrata}. 
Then, look at $\pi^{2(e-1)}\omega_{\tau}^{[1]}/\omega_{\tau}^{[1]}$ : this space has a natural lift $\pi^{2(e-1)}\widetilde{\omega}_\tau^{[1]}/\widetilde{\omega}_\tau^{[1]}$ inside $\widetilde{\mathcal E}/\widetilde{\omega}_{\tau}^{[1]}$. We then take a isotrivial lift of the filtration $\dots \subset \omega_{\tau}^{[i-1]}/\omega_{\tau}^{[1]} \subset \omega_{\tau,1}^{[i]}/\omega_\tau^{[1]} \subset \omega_\tau^{[i]}/\omega_\tau^{[1]} \subset \dots$, for $e \geq i \geq 1$ and then pull back to $\widetilde{\mathcal E} = \mathcal E \otimes_k k((t))$ to get a full lift, and we get a point over $k((t))^{perf}$ with new $\widetilde{c}_\tau^{[1]} = c_\tau^{[1]'}$ but $\widetilde{c}^{[i]}_\tau = c_\tau^{[i]}$ for $i \geq 2$. Then by induction, we can assume that for $1 \leq s \leq i$ we have $c_\tau^{[s]} := (h_\tau^{[s]},l_\tau^{[s]}) = (h_\tau^{[s]'},l_\tau^{[s]'}) =: c_\tau^{[s]'}$ for our point over $k$. Then one deforms $\omega_{\tau,1}^{[i+1]}/\omega_{\tau,1}^{[i]} \subset \omega_\tau^{[i+1]}/\omega_{\tau,1}^{[i]}$ inside $\widetilde{\mathcal E}_\tau^{[i]}$ again using Proposition \ref{propclosurestrata}, and we do the same isotrivial lift for the rest of the filtration as when $i = 0$, to get the induction step, and thus the result.
\end{proof}

\begin{theor}
$Y$ is normal and flat over $\mathcal O_{F_\pi}$, and its special fiber is reduced and Cohen-Macaulay. 
\end{theor}

\begin{proof}
Consider again a local model diagram as in the proof of Proposition \ref{prop:flat0}. The local model splits over direct factors of $B_{\QQ_p}$, thus it is sufficient to show the theorem for one such factor only. Let $F_\pi/F_{\pi_0}, e,f$ etc. as before, $M_n(\mathcal O_{F_{\pi}})$ the factor, and denote $\Lambda'$ the part of $\Lambda \otimes_\ZZ \ZZ_p$ corresponding to this simple factor and using Morita equivalence (so that $\Lambda = \sum_j \mathcal O_{F_{\pi}}^n\otimes_{\mathcal O_{F_{\pi}}} \Lambda'$). We have a diagram
\[ Y \longleftarrow \widetilde{Y} = \Isom(\mathcal E,\Lambda\otimes \mathcal O_S) \fleche \mathcal N,\]
where the first map is a torsor over a smooth group scheme $\mathcal G$, and the second map is formally smooth and $\mathcal G$-equivariant by Grothendieck-Messing. Here $\mathcal N$ is a local model, see e.g. \cite{P-R} section 14, analogous to the one in the proof of Proposition \ref{prop:flat0}, parametrizing 
\begin{itemize}
\item
A PR-filtration $0 = F_{\tau}^{[0]} \subset F_{\tau}^{[1]} \subset \dots \subset F_{\tau}^{[e]} = \Lambda_{\tau,j} \otimes \mathcal O_S$ in $\Lambda \otimes \mathcal O_S$, each $F_\tau^{[i]}$ is a locally direct factor, stable by $\mathcal O_{F_\pi}$.
\item Each quotient $F_{\tau}^{[i]}/F_{\tau}^{[i-1]}$ is locally free of rank $h = a_{\tau,i} + b_{\tau,i}$ and $\mathcal O_{F_{\pi_0}}$ acts by $\sigma_{\tau,i}$ on it.
\item The filtration is compatible with the polarisation
\item For each $i$, a locally direct factor $F_\tau^{[i-1]} \subset F_{\tau,1}^{[i]} \subset F_\tau^{[i]}$, stable by $\mathcal O_{F_\pi}$,
\item $F_{\tau,1}^{[i]}/F_\tau^{[i-1]}$ is a locally direct factor of rank $a_{\tau,i}$ and $\mathcal O_{F_\pi}$ acts through $\sigma_{\tau,i}^+$
\item $\mathcal O_{F_\pi}$ acts through $\sigma_{\tau,i}^-$ on $F_\tau^{[i]}/F_{\tau,1}^{[i]}$ (and this is automatically locally free of rank $b_{\tau,i}$).
\end{itemize}
$F_\tau^{[i]}$ is obviously the analog of the $\omega_\tau^{[i]}$ in the definition of $Y$, and $F_{\tau,1}^{[i]}$ of $\omega_{\tau,1}^{[i]}$. Thus it is enough to see that $\mathcal N$ is flat, normal, and its special fiber is Cohen-Macaulay. The theorem \ref{thm:closurerelations} actually shows that 
\[ N := \overline{\mathcal N} = \coprod_{c} N_c,\]
with expected (strong) closure relations. The proof of Proposition 2.12 for the maximal strata carries over and shows (doing one $F_{\tau,1}^{[i]}$ at a time) that maximal strata of $N$ are smooth, thus reduced, and $\mathcal N$ is smooth in codimension 1. For each $i$, we have a space $\mathcal N_{\leq i}$ parametrizing locally direct factors $F_{\tau}^{[1]} \subset \dots \subset F_\tau^{[i]} \subset \Lambda_{\tau,j} \otimes \mathcal O_S$ with the same properties as before, together with $F_{\tau,1}^{[k]}$ in $F_{\tau}^{[k]}/F_\tau^{[k-1]}$ of rank $a_{\tau,k}$ such that the actions of $\mathcal O_{F_\pi}$ is through $\sigma_{\tau,k}^+$ on $F_{\tau,1}^{[k]} / F_\tau^{[k-1]}$ and by $\sigma_{\tau,k}^-$ on the cokernel of the inclusion, for $k = 1,\dots,i$. We have a natural maps 
\[ \mathcal N = \mathcal N_{\leq e} \fleche \mathcal N_{\leq e-1} \fleche \dots \fleche \mathcal N_{\leq 1} \fleche \Spec(\mathcal O_F) =: \mathcal N_{\leq 0}.\]
We will show inductively that $\mathcal N_{\leq i}$ is flat over $\mathcal O_{F_\pi}$, with Cohen-Macaulay fibers. 

As $\mathcal N$ and $\mathcal N_{\leq i}$ decomposes naturally as product over the simple factors of $B_{\QQ_p}$, and over the index $\tau$, thus we can assume that there is only one factor and that $\mathcal O_B \otimes \ZZ_p = \mathcal O_F$ and that $\mathcal T = \{\tau\}$ so we suppress $\tau$ from the notations. Denote ${E}^{[i]} := (\pi_0 - \sigma_{\tau,i}(\pi_0))^{-1} F^{[i-1]} / F^{[i-1]}$, endowed with its own (perfect) pairing. By definition $\mathcal N_{\leq i}$ over $\mathcal N_{\leq i-1}$ parametrizes locally direct factors $F_1^{[i]}$ and $F^{[i]}$ of $E^{[i]}$ of respective ranks $a_i$ and $h = a_i + b_i$, such that moreover $F_1^{[i]} \subset F^{[i]}$ and $F_1^{[i]} \subset E^{[i]}[\pi - \sigma_i^+(\pi)]$, and a compatibility for the polarisation. We assume that $\mathcal N_{i-1}$ is flat over $\mathcal O_{F_\pi}$,  with Cohen-Macaulay fibers. So let $U \subset \mathcal N_{\leq i-1}$ for $i \geq 1$ a small affine so that all $F^{[k]}, k < i$ and ${E}^{[i]} := (\pi_0 - \sigma_{\tau,i}(\pi_0))^{-1} F^{[i-1]} / F^{[i-1]}$ are free. Now we claim that we can make the pairing of $E^{[i]}$ locally trivial. First, at it is perfect, we choose a basis so that it is of the form
\[
\left(
\begin{array}{cc}
 0 &   
\left(
\begin{array}{ccc}
 a_1 &   &   \\
  &   \ddots &   \\
  &   &   a_h
\end{array}
\right)
   \\
\left(
\begin{array}{ccc}- a_1 &   &   \\
  &   \ddots &   \\
  &   &   -a_h
\end{array}
\right)  &   0
\end{array}
\right),
\]
with $a_i \in \mathcal O_S^\times$. But now, up to changing the basis vectors $(e_1,\dots,e_h,f_1,\dots,f_h)$ by \\
$(e_1,\dots,e_h,a_1^{-1}f_1,\dots,e_h^{-1}f_h)$, it is of the desired form
\[
\left(
\begin{array}{cc}
 0 &   
\left(
\begin{array}{ccc}
 1 &   &   \\
  &   \ddots &   \\
  &   &   1\end{array}
\right)
   \\
\left(
\begin{array}{ccc}- 1 &   &   \\
  &   \ddots &   \\
  &   &   -1
  \end{array}
\right)  &   0
\end{array}
\right),
\]
and thus $\mathcal N_{\leq i}\times_{\mathcal N_{\leq i-1}} U \simeq U \times_{\mathcal O_{F_\pi}} \mathcal N'$ where $\mathcal N'$ classifies $\mathcal O_{F_{\pi_0}}$ locally direct factors
$(\mathcal F_1,\mathcal F)$ inside $\mathcal O_{F_\pi}^{2h} \simeq \mathcal O_{F_{\pi_0}}[X]/(X^2- \pi_0)^{2h}$ of respective ranks $a_i,a_i+b_i$ satisfying relations analogous to the one of $\mathcal N$ in the proof of Proposition \ref{prop:flat0}. In particular, we can check that changing $\ZZ_p$ by $\mathcal O_{F_{\pi_0}}$ (and thus $p$ by $\pi_0$) in the proof there we have that $\mathcal N'$ is flat over $\mathcal O_{F_\pi}$, with reduced Cohen-Macaulay fibers. Thus $\mathcal N' \times U$ is flat and Cohen-Macaulay over $U$, thus $\pi_i : \mathcal N_{\leq i} \fleche \mathcal N_{i-1}$ is flat and Cohen-Macaulay (with reduced fibers). Thus $\mathcal N_{\leq i}$ is flat and Cohen-Macaulay over $\mathcal O_{F_\pi}$ by the induction hypothesis and by induction $\mathcal N \fleche \mathcal O_{F_\pi}$ is flat with (reduced) Cohen-Macaulay fibers. Moreover, $ \mathcal N$ is normal being smooth in codimension 1.

\end{proof}

\section{The case when $p=2$.}
\label{sect:6}
In this section, we will investigate the case $p=2$.

\subsection{Quadratic forms in characteristic $2$}

Let $k$ be an algebraically closed field of characteristic $2$, let $V$ be a $k$-vector space of dimension $d$. Let $<,>$ be a non-degenerate symmetric bilinear form, and let $q$ be the associated quadratic form defined by $q(x) = <x,x>$.

\begin{prop}
\label{propformquadcar2}
Up to isomorphism, we are in one of the two following situations:
\begin{enumerate}
\item $q$ is not identically zero, and the matrix of the bilinear form is the identity matrix in a certain basis.
\item $q$ is identically zero. This implies that $d$ is even, and the matrix of the bilinear form in a certain basis is of the form
$$\left( \begin{array}{cccc}
A & 0 & \dots & 0 \\
0 & A & \ddots & \vdots \\
\vdots & \ddots & \ddots & 0 \\
0 & \dots & 0 & A
\end{array} \right) \qquad A = \left ( \begin{array}{cc}
0 & 1 \\
1 & 0
\end{array} \right )$$
\end{enumerate}
\end{prop}

\begin{proof}
One proves this result by induction on the dimension. If the dimension is $1$ or $2$, it is an easy computation. \\
Assume the result true for all $k \leq d-1$, and let us prove it for $d$. Assume that $q$ is identically zero. Take a vector $e_1$ in $V$ and a vector $e_2$ such that $<e_1, e_2> = 1$. Let $F$ be the orthogonal of the vector space spanned by $e_1, e_2$. Applying the induction hypothesis to $F$ gives the result. \\
Assume now that $q$ is not identically $0$. Let $e_1$ be a vector, normalized such that $<e_1, e_1> = 1$. Let $F$ be the orthogonal of the space generated by $e_1$. One can apply the induction result to $F$. This gives the result, noticing that the matrices
$$ \left ( \begin{array}{ccc}
1 & 0 & 0 \\
0 & 0 & 1 \\
0 & 1 & 0
\end{array} \right ) \qquad \left ( \begin{array}{ccc}
1 & 0 & 0 \\
0 & 1 & 0 \\
0 & 0 & 1
\end{array} \right )$$
are equivalent. Indeed, if $e_1, e_2, e_3$ is a basis for which the matrix is the second one, then the change of basis $e_1' = e_1 + e_2 + e_3$, $e_2' = e_1 + e_2$, $e_3' = e_2 + e_3$ gives the first matrix.
\end{proof}

\subsection{Geometry in the first case}

In this section we assume that we are in the first case, i.e. the modified pairing given on $A[\pi]$ is given by the identity matrix.

\begin{prop}
The smooth locus is $X_{0,0}$, which has dimension $ab$. Moreover, the other strata $X_{(h,\ell)}$, $\ell \neq 0$, are non-smooth.
\end{prop}

\begin{proof}
It is clear that the points in $X_{0,0}$ are smooth as this is open in $X$, and square-zero deformations corresponds to deforming in a Grassmanian $\mathcal{G}r_{a,a+b}$, this gives also the dimension. To prove that these are the only smooth points, we'll argue as the case of characteristics $p \neq 2$, but first we need a lemma.

\begin{claim}
Let $V,(,)$ be a non-degenerate symmetric space of rank $N$ over $k$ algebraically closed of characteristics 2 with quadratic form $q$ non-zero. Let $W \subset V$ totally isotropic of rank $h$. Then we claim that there exists a basis $\underline e$ of $V$ such that the matrix of $(,)$ is $I_n$ and $W = \Vect(e_1+e_2,e_3+e_4,\dots,e_{2h-1}+e_{2h})$.
\end{claim}

\begin{proof}[Proof of claim]
Indeed, let $W = \Vect(f_1,\dots,f_h)$ and $g_1,\dots,g_h$ such that $(g_i,f_j) = \delta_{i,j}$ (this is possible by pulling back the dual basis of a completion of $\underline f$). We claim that we can modify the $g_i$ so that $(g_i,g_j) = \delta_{i,j}$. If $(g_1,g_1)$ is non zero then we can rescale to get $(g_1,g_1) = 1$. Otherwise, there exists $v \in (f_1,\dots,f_h)^\perp$ such that $(v,v) \neq 0$ : indeed, if $(v,v) = 0$ then $(f_1,\dots,f_h)$ is totally isotropic, thus there can exists at most $N/2 - h$ such vectors, but $\dim(f_1,\dots,f_h)^\perp = N-h$. Setting $g_1+v$ instead of $g_1$ we do not change the values on the $f_i$'s but $(g_1+v,g_1+v) = (v,v)$ as we are in characteristics 2. Assume that we have constructed $g_1,\dots,g_i, i < h$ such that $(g_k,f_i) = \delta_{k,i}$ for all $k,i \leq h$ and $(g_k,g_\ell) = \delta_{k,l}$ for all $k,\ell \leq i$.
We clam that we can find $v \in T = (f_1,\dots,f_h,g_1,\dots,g_i)^\perp$ such that $(v,v) \neq 0$. This space is $N-h-i$-dimensional. If a basis $v_1,\dots,v_{N-h-i}$ of vectors for this space are of norm 0; then they are also orthogonal since $2 = 0$. In particular, $(v_1,\dots,v_{N-h-i})^\perp$ contains $(f_1,\dots,f_h,g_1,\dots,g_i,v_1,\dots,v_{N-h-i})$. This last space has dimension $N$, thus $T = 0$, which is absurd since $i < h \leq N/2$. Choose a $v$ of non zero norm, then set
\[ g_{i+1}' = g_i + \sum_{k=1}^i \lambda_k f_k + v.\]
The norm of $g_{i+1}'$ is then $(g_{i+1},g_{i+1}) +  (v,v)$ and, for $k \leq i$,
\[ (g_{i+1},f_k) = 0, \quad (g_{i+1},g_k) = (g_{i+1},g_k) + \lambda_k.\]
Thus if we set $\lambda_k = - (g_{i+1},g_k)$ we have that $(g_{i+1}',g_k) = 0$ and, up to change $v$ by $\mu v, \mu \in k$ and rescaling $g_{i+1}$, we can assume that 
$(g_{i+1},g_{i+1}) = 1$. Setting $e_{2i-1} = f_i - g_i$ and $e_{2i} = g_i$, we have a beginning of a basis such that $(e_i,e_j) = \delta_{i,j}$. Then, if $h < N/2$, we can look at $\Vect(e_1,\dots,e_{2h})^\perp$ and argue as the end of proof of proposition \ref{propformquadcar2}
\end{proof}

Now let $x \in X_{h,\ell}(k)$ with $\ell \neq 0$. Let us prove that the strata are not smooth at any closed point. Then a $k[t]/(t^n)$ lift of $x$ in $X_{h,\ell}$ induces a lift of $W = \omega_1 \cap \omega_2 \subset \mathcal E[\pi]$, which is totally isotropic. We will find a lift of $x$ mod $t^2$ which we cannot lift. First, by the claim we can assume that the matrix of the divided pairing is the identity, and $W$ is given in the basis $(e_1,e_3,\dots,e_{2h-1},e_{2},e_4,\dots,e_{2h})$ by
\[ \left(\begin{array}{c} I_h \\ I_h \end{array}\right).\]
Now we will choose the deformation such that the lift of $W$ is given by
\[ \left(\begin{array}{c} I_h \\ I_h \end{array}\right) + t\left(\begin{array}{c} N_1 \\ N_2 \end{array}\right),\]
and we will actually set $N_1 = E_{1,1} + tN_1'$, and $N_2 = t N_2'$. Then the lift of $W$ is totally isotropic if 
\[ t(N_1 + ^{t}N_1) + t^2 {^t}N_1N_1 + t(N_2 + {^t}N_2) + t^2 {^t}N_2N_2 = 0,\]
thus the lift of $W$ mod $t^2$ is indeed totally isotropic, and we check there is indeed a lift to $k[t]/(t^2)$ which gives this lift of $W$\footnote{deform $\pi^{-1}(e_1+e_2) \in \omega$ by $\pi^{-1}(e_1+e_2) + t\pi^{-1}e_1$ and check it remains totally isotropic mod $t^2$}. Now we can show that there is no choice of $N_1',N_2'$ such that this lift to $k[t]/(t^3)$. Indeed, otherwise we would get as equation mod $t^3$
\[ t^2 E_{1,1} + t^2(N_1' + {^t}N_1' + N_2' + {^t}N_2') = 0,\]
but we can check easily that the right matrix always has a zero coefficient in position $(1,1)$ as we are in characteristics 2. This imply that all strata $X_{h,\ell}, \ell \neq 0$ aren't smooth at any point. In particular, as they are open, the strata $X_{h,h}$ are not is the smooth locus. Let us prove the following claim
\begin{claim}
Any point $x \in X_{h,\ell}(k)$ with $\ell \neq 0$ can deformed in $X_{h',\ell}(k)$ with $h' \neq 0$.
\end{claim}

\begin{proof}[Proof of claim]
If $h \neq 0$ this is trivial thus assume $h =0$. We will construct a $k[[t]]$-deformation of $x$ whose generic fiber lies in $X_{1,\ell}$. Choose any basis $\pi e_1,\dots, \pi e_n$ of $\omega = \mathcal E[\pi]$ such that $\pi e_1,\dots, \pi e_\ell$ is a basis of $\omega_1\cap \omega_2$, $\pi e_1,\dots,\pi e_a$ of $\omega_1$ and $\pi e_1,\pi e_\ell,\pi e_{a+1},\dots,\pi e_{a+b-\ell}$ of $\omega_2$. As $\ell  > 0$ we have $\omega_1 + \omega_2 \subsetneq \omega = \mathcal E[\pi]$, thus we can assume $\pi e_n \notin \omega_1+\omega_2$. Let $\widetilde{\mathcal E} = \mathcal E \otimes_k k[[t]]$. Let $\pi e_1,\dots,\pi e_n$ denote the $k[[t]]$-basis of $\widetilde{\mathcal E}$ with the same pairing matrix. Set $\widetilde \omega_1$ generated by $\pi e_1,\dots,\pi e_a$, $\widetilde \omega_2$ by $\pi e_1,\pi e_\ell,\pi e_{a+1},\dots,\pi e_{a+b-\ell}$ and $\widetilde\omega$ by $\pi e_1,\dots,\pi e_{n-1}, \pi e_n + t e_1$, for a preimage by $\pi$ of $e_1$ in $\widetilde{\mathcal E}$. We then check that $\widetilde \omega$ is totally isotropic : $<\pi e_n + te_1,\pi e_n + te_1> = <\pi e_n,te_1> + <te_1,\pi e_n> = 0$. Reducing all these data to $k[t]/(t^2)$ (which has divided powers over $k$) we deduce by Grothendieck-Messing and Serre-Tate a deformation of $x$ to $k[t]/(t^2)$, and then inductively for all $n$ to $k[t]/(t^n)$ (reducing this construction over $k[[t]]$), thus we get a $k[[t]]$-point of $X$, whose generic fiber has $h = 1$.
\end{proof}

Now, as the smooth locus of $X$ is open, it cannot contain any of the $X_{h,\ell}$ with $\ell \neq 0$. Indeed, if $x \in X_{h,\ell}$ with $\ell \neq 0$ is in the smooth locus, any deformation of it also is in the smooth locus. But there is a deformation with $h \neq 0$ thus we can assume that $h \neq 0$ for $x$. Then any deformation of $x$ has $h \neq 0$ and take a deformation $y$ of $x$ with maximal $h$ and minimal $\ell \geq h$, thus $\ell \neq 0$, and $y$ is still in the smooth locus. Up to change $x$ by $y$, we can assume every deformation of $x$ lies in $X_{h,\ell}$. Because the smooth locus is open, there is an open $U \subset X^{sm}$ containing $Y$. Reducing $U$ if necessary, we can assume that $U$ is irreducible and is included in $\cup_{h' \geq h,\ell' \leq \ell} X_{h',\ell'}$. But as there is no further possible deformation of $x$ in $X$, actually $U$ is included in $X_{h,\ell}$. But the smooth locus of $X_{h,\ell}$ is empty as $\ell \neq 0$\footnote{Obviously if we prove the closure relations for the strata, then this argument simplifies a lot.}.
\end{proof}

\subsection{Geometry in the second case}

In this section we assume that we are in the second case, i.e. the modified pairing given on $A[\pi]$ is given by the matrix
$$\left( \begin{array}{cccc}
A & 0 & \dots & 0 \\
0 & A & \ddots & \vdots \\
\vdots & \ddots & \ddots & 0 \\
0 & \dots & 0 & A
\end{array} \right) \qquad A = \left ( \begin{array}{cc}
0 & 1 \\
1 & 0
\end{array} \right )$$

\begin{prop}
\label{prop:6.5}
The stratum $X_{h,l}$ is empty if $l$ is not equal to $a$ modulo $2$.
\end{prop}

\begin{proof}
Let us consider the modified pairing on $\omega_1$. It induces a non degenerate bilinear form on $\omega_1 / \omega_1 \cap \omega_2$. The associated quadratic form on this space is identically zero; this implies that its dimension must be even. Since the dimension is equal to $a-l$, the result follows.
\end{proof}

\begin{prop}
The smooth locus consists in the strata $X_{h,h} \cup X_{h-1,h}$ for $1 \leq h \leq a$, with $h =a$ modulo $2$, and $X_{0,0}$ if $a$ is even. Each of the previous sets are open, of dimension $ab + h$ and $ab$ respectively.
\end{prop}

\begin{proof}
It is clear that the points in $X_{0,0}$ are smooth. Let $1 \leq h \leq a$, with $a-h$ even. By the previous proposition, it is clear that $X_{h,h} \cup X_{h-1,h}$ is open. Let $x$ be a point in $X_{h,h} \cup X_{h-1,h}$, and let us prove that it is a smooth point, by computing the tangent space. First let us remark that on $X_{h,h} \cup X_{h-1,h}$, the space $\omega_1 \cap \omega_2$ has rank $h$. Deforming $x$ thus amounts to first deform the space $\omega_1 \cap \omega_2$ to a totally isotropic space $\mathcal{F}$; then deform $\omega_1$ inside the orthogonal of $\mathcal{F}$. Finally, one should deform $\omega$, contained in $\pi^{-1} \mathcal{F}$ and containing $\mathcal{F}^\bot$. Note that the original pairing descends to the quotient $\pi^{-1} \mathcal{F} / \mathcal{F}^\bot$. The second and third operations are smooth, of dimension respectively $(a-h)(b-h)$ and $h(h+1)/2$. We thus need to prove that the first operation is smooth of dimension $h(a+b-h) - h(h-1)/2$. \\
One can assume that the matrix of the modified pairing on $\mathcal E[\pi]$ is
$$\left( \begin{array}{ccc}
0 & 0 &  I_h \\
0 & B &  0 \\
I_h  & 0 & 0
\end{array} \right) $$
where this decomposition is written with respect to the inclusions $\omega_1 \cap \omega_2 \subseteq \omega_1 + \omega_2$, and $B$ is a matrix with copies of the matrix $A$ on the diagonal. Deforming the space $\omega_1\cap \omega_2$ involves a matrix $\left( \begin{array}{c}
I_h \\
X \\
Y
\end{array} \right)$, and thus $h(a+b-h)$ coordinates. The fact that this space should be totally isotropic gives the condition
$$Y + ^tY = ^tX B X$$
This is an equality between symmetric matrices which have coefficients $0$ on the diagonal. This is thus a smooth condition, with $h(h-1)/2$ linearly independent equations. We are thus left to prove that any point not in $\cup_{h \equiv a \pmod 2} X_{h,h} \cup X_{h-1,h}$ is not a smooth point. Let $x \in X_{h,\ell}(k)$ with $\ell - k \geq 2$. In particular $\dim \omega_1\cap \omega_2 / (\pi \omega) \geq 2$, thus let $\pi e_1,\pi e_2 \in \omega_1\cap\omega_2 \subset \mathcal E[\pi]$ be two vectors, linearly independant when sent to $\omega_1\cap \omega_2/(\pi \omega)$. As in the case of characteristics 2, we will look for specific lifts. As $\dim \omega[\pi] \backslash (\omega_1+\omega_2) = a+b-h-(a+b-\ell) = \ell - h \geq 2$, we can find two linearly independant vectors $\pi e_{a+b-h-1},\pi e_{a+b-h}$ there such that $\{ \pi e_1,\pi e_{a+b-h-1}\} = 1 = \{ \pi e_2,\pi e_{a+b-h}\}$ and $\{ \pi e_1,\pi e_{a+b-h} \} = 0 = \{ \pi e_2,\pi e_{a+b-h-1}\}$. Indeed, $\omega[\pi]$ is the orthogonal of $\pi \omega$ for the modified pairing $\{ , \}$. Moreover modifiying $\pi e_{a+b-h-1}$ by $\pi e_{a+b-h-1} - t \pi e_{a+b-h}$ we can moreover assume $\{\pi e_{a+b-h-1},\pi e_{a+b-h}\} = 0$ (the norm are automatically zero as we are in the second case). Assume that $\pi e_1,\dots,\pi e_n$ is a basis of $\mathcal E$ such that  $\pi e_1,\dots,\pi e_a$ is a basis of $\omega_1$, $\pi e_1,\dots,\pi e_\ell,\pi e_{a+1},\dots,\pi e_{a+b-\ell}$ a basis of $\omega_2$, and $\pi e_1,\dots,\pi e_{a+b-h}$ a basis of $\omega_\pi$, and assume given lifts $e_{\ell},\dots,e_{\ell-h+1}$ of $\pi e_{\ell},\dots,\pi e_{\ell-h+1}$ inducing a basis of $\omega/\omega[\pi]$ (they are thus two by two orthogonal). We can also assume that $\{\pi e_{a+b-h-1},\pi e_j\} = 0$ for all $j \neq 1$ and $\{\pi e_{a+b-h},\pi e_j\} = 0$ for all $j \neq 2$ (up to modify by linear combination of the $\pi e_j$). Choose $e_1,e_2,e_{a+b-h-1},e_{a+b-h} \in \mathcal E$ which are sent by $\pi$ to $\pi e_1,\pi  e_2,\pi e_{a+b-h-1},\pi e_{a+b-h}$ and which are moreover two by two orthogonal together with $e_{\ell},\dots,e_{\ell-h+1}$ (we can modify each by a $\pi$-torsion element).
We set the following in $\widetilde{\mathcal E} = \mathcal E \otimes_k k[t]/(t^3)$ :
\[ \widetilde\omega_1 = (\pi e_1,\pi e_2 + t \pi e_{a+b-h-1},\pi e_3,\dots,\pi e_a), \quad \text{and} \]\[ \widetilde \omega = (\pi e_1,\dots, \pi e_{a+b-h-2},e_{\ell},\dots,e_{\ell-h+1},\pi e_{a+b-h-1} + t e_1, \pi e_{a+b-h} + t e_2 + t^2 e_{a+b-h-1}\]
Clearly, $\pi \widetilde\omega \subset \widetilde\omega_1$. We claim that modulo $t^2$ this defines a lift of $x$. We just need to check if $\omega$ is totally isotropic, and this boils down to
\[ < \pi e_2 + t \pi e_{a+b-h-1}, \pi e_{a+b-h} + t e_2> = \{ \pi e_2 + t \pi e_{a+b-h-1}, t \pi e_2\} = 0, \]
\[ < \pi e_1, \pi e_{a+b-h-1} + t e_1> = \{ \pi e_1 , t \pi e_1\} = 0, \]
\[< \pi e_2 + t \pi e_{a+b-h-1}, \pi e_{a+b-h-1} + t e_1> = \{ \pi e_2 + t \pi e_{a+b-h-1}, t \pi e_1\} = t^2 = 0, \]
\[< \pi e_1 , \pi e_{a+b-h} + t e_2> = \{ \pi e_1, t \pi e_2\} =  0, \]
Now assume we have another lift modulo $t^3$, this implies that there exists vectors $v_1 \in \widetilde{\omega_1}'$ with $v_1 = \pi e_1 + t^2 w_1$, $v_2 =  \pi e_2 + t \pi e_{a+b-h-1} + t^2w_2$, and $v_3,v_4\in \widetilde\omega'$ such that $v_3 = \pi e_{a+b-h-1} + t e_1 + t^2 w_3, v_4 = \pi e_{a+b-h} + t e_2 + t^2 e_{a+b-h-1}+t^2 w_4$. Moreover, $t^2\pi w_3,t^2\pi w_4$. But then, $\widetilde{\omega}'$ should be totally isotropic, but,
\begin{eqnarray*} 
<v_2,v_3> = < \pi e_2 + t \pi e_{a+b-h-1} + t^2w_2, \pi e_{a+b-h-1} + t e_1 + t^2w_3> \\ 
= \{ \pi e_2 + t \pi e_{a+b-h-1} + t^2 v_2, t \pi e_1 + t^2 \pi w_3\} = t^2 \neq 0. \end{eqnarray*}

\end{proof}

\begin{rema}
If $a=1$, the whole variety is smooth.
\end{rema}

\bibliographystyle{smfalpha} 
\bibliography{biblio} 

\end{document}